\documentclass[compmod]{rsl}

\gridframe{N}

\usepackage{proof}
\usepackage{amsmath,amssymb,latexsym,natbib}
\usepackage{rotating}

\renewcommand{\thesection}{\S\arabic{section}}

\volume{0}
\issue{0}

\pyear{2020}
\pmonth{---}
\doinu{10.1017/S1755020300000000}
\title[Lambek Calculus with Kleene Star]{Complexity of the Infinitary Lambek Calculus with Kleene Star}
\author[S. Kuznetsov]{STEPAN KUZNETSOV}
\affil{Steklov Mathematical Institute of the RAS}

\leftrunninghead{stepan kuznetsov}
\rightrunninghead{Lambek Calculus with Kleene Star}

\newcommand{\BS}{\mathop{\backslash}}
\newcommand{\SL}{\mathop{/}}
\newcommand{\KStar}{{}^*}

\newcommand{\LL}{\mathbf{L}\!^{\Lambda}}

\newcommand{\Lpomega}{\mathbf{L}_\omega^{\!+}}

\newcommand{\Lomega}{\mathbf{L}_\omega^{\!*}}

\newcommand{\One}{\mathbf{1}}

\newcommand{\nrm}[1]{[\![#1]\!]}

\newcommand{\Uc}{\mathcal{U}}
\newcommand{\is}{\mathrm{is}}

\newcommand{\Nc}{\mathcal{N}}
\newcommand{\Gc}{\mathcal{G}}

\newcommand{\Var}{\mathrm{Var}}
\newcommand{\FG}{\mathbb{FG}}

\newcommand{\CUT}{\mathrm{cut}}

\newcommand{\Af}{\mathfrak{A}}

\newcommand{\LU}{\mathbf{L}_{\One}}

\newcommand{\ACTomega}{\mathbf{ACT}_\omega}

\newcommand{\ThM}{\mathrm{Th}_{\BS,\SL,\cdot,\KStar}}
\newcommand{\ThA}{\mathrm{Th}_{\BS,\SL,\cdot,\KStar,\vee,\wedge}}
\newcommand{\LMod}{\mathrm{L\mbox{-}Mod}}
\newcommand{\RMod}{\mathrm{R\mbox{-}Mod}}
\newcommand{\REGLANMod}{\mathrm{REGLAN\mbox{-}Mod}}

\newcommand{\Inst}{\mathrm{Inst}}

\begin{document}

\maketitle

\begin{abstract}
We consider the Lambek calculus, or non-commutative multiplicative intuitionistic linear logic,
 extended with iteration, or Kleene star, axiomatised by means of an $\omega$-rule, and prove that the derivability problem in this calculus
is $\Pi_1^0$-hard. This solves a problem left open by Buszkowski (2007), who obtained the same complexity bound for infinitary action logic, 
which additionally includes additive conjunction and disjunction. 
As a by-product, we prove that any context-free language without the empty word can be generated by a Lambek grammar
with unique type assignment, without Lambek's non-emptiness restriction imposed (cf. Safiullin 2007). 


\end{abstract}

\section{Introduction}\label{S:intro}
%
%
%
Residuated structures play an important r\^{o}le in abstract algebra and substructural 
logic~\citep{Krull1924,WardDilworth1939,Ono1993,JipsenSurvey,GalatosRLbook,AbramskyTzevelekos2010}.
We introduce the Lambek calculus with the unit~\citep{Lambek1969}
as the algebraic logic (inequational theory) of {\em residuated monoids}.
A residuated monoid is a partially ordered algebraic structure $\langle \Af; \preceq, \cdot, \One, \BS, \SL \rangle$, where:
\begin{itemize}
\item $\langle \Af; \cdot, \One \rangle$ is a monoid;
\item $\preceq$ is a preorder;
\item $\BS$ and $\SL$ are {\em residuals} of $\cdot$ w.r.t. $\preceq$, i.e.
$$A \preceq C \SL B \iff A \cdot B \preceq C \iff B \preceq A \BS C.$$
\end{itemize}

Notice that in the presence of residuals we do not need to postulate monotonicity of $\cdot$ w.r.t. $\preceq$ explicitly:
it follows from the rules for $\BS$ and $\SL$~\citep{Lambek1958}.

The Lambek calculus with the unit, $\LU$,  axiomatises the set of atomic sentences of the form $A \preceq B$ (where $A$ and $B$ are formulae constructed from variables
and constant $\One$ using three binary operations: $\cdot$, $\BS$, $\SL$) which are generally true in residuated monoids.
 
 We formulate $\LU$ in the form of a Gentzen-style sequent calculus. 
 Formulae of $\LU$ are built from a countable set of variables
$\Var = \{ p_1, p_2, p_3, \ldots \}$ and constant $\One$ using three binary connectives: $\BS$, $\SL$, $\cdot$\,. {\em Sequents} are expressions of the
form $\Pi \to A$, where $A$ (succedent) is a formula and $\Pi$ (antecedent) is a finite sequence of formulae. 
The antecedent $\Pi$ is allowed to be empty; the empty sequence of formulae is denoted by $\Lambda$.

Sequent $B_1, \ldots, B_n \to A$ are interpreted as $B_1 \cdot \ldots \cdot B_n \preceq A$; sequent $\Lambda \to A$ means $\One \preceq A$.

Axioms of $\LU$ are sequents of the form $A \to A$, and $\Lambda \to \One$. Rules of inference are as follows:
\clearpage
$$
\infer[(\BS\to)]{\Gamma, \Pi, A \BS B, \Delta \to C}{\Pi \to A & \Gamma, B, \Delta \to C}
\qquad
\infer[(\to\BS)]{\Pi \to A \BS B}{A, \Pi \to B}
$$
$$
\infer[(\SL\to)]{\Gamma,  B \SL A, \Pi, \Delta \to C}{\Pi \to A & \Gamma, B, \Delta \to C}
\qquad
\infer[(\to\SL)]{\Pi \to B \SL A}{\Pi, A \to B}
$$
$$
\infer[(\cdot\to)]{\Gamma, A \cdot B, \Delta \to C}{\Gamma, A, B, \Delta \to C}
\qquad
\infer[(\to\cdot)]{\Gamma, \Delta \to A \cdot B}{\Gamma \to A & \Delta \to B}
\qquad
\infer[(\One\to)]{\Gamma, \One, \Delta \to C}{\Gamma, \Delta \to C}
$$
$$
\infer[(\mathrm{cut})]{\Gamma, \Pi, \Delta \to C}{\Pi \to A & \Gamma, A, \Delta \to C}
$$

Completeness is proved by standard Lindenbaum -- Tarski construction.

One of the natural examples of residuated monoids is the algebra $\mathcal{P}(\Sigma^*)$ of formal languages  over
an alphabet $\Sigma$. The preorder on $\mathcal{P}(\Sigma^*)$ is the subset relation; multiplication is pairwise
concatenation:
$$
A \cdot B = \{ uv \mid u \in A, v \in B \}
$$ 
and divisions are defined as follows:
\begin{align*}
A \BS B &= \{ u \in \Sigma^* \mid (\forall v \in A)\, vu \in B \};\\
B \SL A &= \{ u \in \Sigma^* \mid (\forall v \in A)\, uv \in B \}.
\end{align*}

This interpretation of the Lambek calculus on formal languages corresponds to the original idea of~\citet{Lambek1958} to
use the Lambek calculus as a basis for {\em categorial grammars.} The concept of categorial grammars goes back to~\citet{Ajdukiewicz}
and~\citet{BarHillel}. Nowadays categorial grammars based on the Lambek calculus and its extensions are used to describe
fragments of natural language in the type-logical linguistic paradigm. In this article we need Lambek categorial
grammars only as a technical gadget for our complexity proofs. Thus, in~\ref{S:grammar} we give only formal definitions and
formulate the results we need; for in-depth discussion of linguistic applications we redirect the reader to~\citet{Carpenter}, \citet{Morrill2011},
and~\citet{MootRetore}.

The notion of residuated monoid can be extended by additional algebraic operations. These include meet ($\wedge$) and
join ($\vee$), which impose a lattice structure on the given preorder: 
$a \wedge b = \inf_{\preceq} \{ a, b \}$, $a \vee b = \sup_{\preceq} \{ a, b \}$, and we postulate that these suprema and infima exist for any $a,b$. A residuated
monoid with meets and joins is a {\em residuated lattice}. On the algebra of formal languages,
meet and join correspond to set-theoretic intersection and union respectively.

On the logical side, meet and join are called, respectively, additive conjunction and disjunction. This terminology follows
Girard's linear logic~\citep{Girard1987}. The monoidal product, $\cdot$, is multiplicative conjunction, and two divisions,
$\BS$ and $\SL$, are left and right linear implications. 
The rules for $\wedge$ and $\vee$ are as follows:
$$
\infer[(\wedge\to),\ i=1,2]{\Gamma, A_1 \wedge A_2, \Delta \to C}{\Gamma, A_i, \Delta \to C}
\qquad
\infer[(\to\wedge)]{\Pi \to A_1 \wedge A_2}{\Pi \to A_1 & \Pi \to A_2}
$$
$$
\infer[(\vee\to)]{\Gamma, A_1 \vee A_2, \Delta \to C}{\Gamma, A_1, \Delta \to C & \Gamma, A_2, \Delta \to C}
\qquad
\infer[(\to\vee),\ i=1,2]{\Pi \to A_1 \vee A_2}{\Pi \to A_i}
$$

The system presented above is substructural, that is, it lacks usual logical principles of weakening, contraction, and permutation.
For this reason, logical connectives split, and we have to consider two implications (left and right) and two versions of conjunction
(multiplicative and additive). For more discussion of substructurality we refer to~\citet{Restall2000}.

Another, more sophisticated operation to be added to residuated monoids is {\em iteration,} or Kleene star, first introduced by~\citet{Kleene1956}. 
Residuated lattices extended with Kleene star are called residuated Kleene lattices (RKLs), or
action lattices~\citep{Pratt1991,Kozen1994}. Throughout this article, we consider only \mbox{\em *-continuous} RKLs, in which Kleene star is defined
as follows:
$$
a^* = \sup\limits_{\preceq} \{ a^n \mid n \geq 0 \}
$$
(where $a^n = a \cdot \ldots \cdot a$, $n$ times, and $a^0 = \One$).\footnote{In the presence of divisions,
the usual definition of *-continuity, $b \cdot a^* \cdot c = \sup_{\preceq} \{ b \cdot a^n \cdot c \mid n \geq 0 \}$, can
be simplified by removing the context $b$, $c$.} In particular, the definition of a *-continuous RKL postulates existence of all these suprema.

Axioms and rules for Kleene star naturally come from its definition: 
$$
\infer[(\KStar\to)_\omega]{\Gamma, A^*, \Delta \to C}{\bigl( \Gamma, A^n, \Delta \to C \bigr)_{n=0}^\infty}
\qquad
\infer[(\to\KStar),\ n \ge 0]{\Pi_1, \ldots, \Pi_n \to A^*}{\Pi_1 \to A & \ldots & \Pi_n \to A}
$$
Here $A^n$ means $A, \ldots, A$, $n$ times; $A^0 = \Lambda$.

The left rule for Kleene star is an $\omega$-rule, {\em i.e.,} has countably many premises. In the presence of $\omega$-rule,
the notion of derivation should be formulated more accurately. A valid derivation is allowed to be infinite, but still should
be well-founded, that is, should not include infinite paths. Thus, the set of theorems of $\ACTomega$ is the smallest set of
sequents which includes all axioms and is closed under application of rules.

Notice that meets and joins are not necessary for defining the Kleene star; thus we can consider {\em residuated monoids with iteration}.

The logic presented above axiomatises the inequational theory of *-continous RKLs.
It is called {\em infinitary action logic}~\citep{BuszkowskiPalka2008} and denoted by $\ACTomega$.
Syntactically, the set of theorems of $\ACTomega$ is the smallest set that includes all axioms and is closed under inference rules presented above;
completeness is again by Lindenbaum -- Tarski construction.

\citet{Palka2007}~proved cut elimination for $\ACTomega$. As usual, cut elimination yields subformula property (if a formula appears somewhere in a cut-free
derivation, then it is a subformula of the goal sequent). Therefore, elementary fragments of $\ACTomega$ with restricted sets of connectives are obtained by
simply taking the corresponding subsets of inference rules. These fragments are denoted by listing the connectives in parentheses after the name of the calculus,
like $\ACTomega(\BS, \vee, \KStar)$, for example. 

Some of these fragments also have their specific names: $\ACTomega(\cdot, \BS, \SL, \One, \KStar)$ is the Lambek calculus with the unit and iteration and is denoted
by $\mathbf{L}^*_{\One\omega}$; its fragment without the unit, $\ACTomega(\cdot, \BS, \SL, \KStar)$, is denoted by $\mathbf{L}^*_\omega$. The multiplicative-additive
Lambek calculus, $\mathbf{MALC}$, is $\ACTomega(\cdot, \BS, \SL, \vee, \wedge)$. The Lambek calculus with the unit, $\LU$, is $\ACTomega(\cdot, \BS, \SL, \One)$.

Finally, $\ACTomega(\cdot, \BS, \SL)$ is the Lambek calculus allowing empty antecedents, but without the unit constant.
 This calculus is usually denoted by $\mathbf{L}^*$~\citep{Lambek1961}.
Unfortunately, this yields a notation clash with the Kleene star, which is also denoted by $\KStar$. 
Just ``$\mathbf{L}$'' is reserved for the system with Lambek's restriction (see the next section). Therefore, we introduce a new name for this calculus:
$\LL$,  which means that empty antecedents are allowed in this calculus. By $\LL(\BS,\SL)$ we denote the product-free fragment of $\LL$ (which is the same as
$\ACTomega(\BS,\SL)$).

\citet{Buszkowski2007} and \citet{Palka2007} show that the derivability problem in $\ACTomega$ is $\Pi_1^0$-complete---in particular, the set of
all theorems of $\ACTomega$ is not recursively enumerable. Thus, the usage of an infinitary proof system becomes inevitable.
\citet{Buszkowski2007} also shows $\Pi_1^0$-hardness for fragments of $\ACTomega$ where one of the additive connectives ($\vee$ or $\wedge$, but not both)
is removed. Complexity of the Lambek calculus with Kleene star, but without both additives, {\em i.e.,} the logic of residuated monoids
with iteration, which we denote by $\Lomega$, however, is left by Buszkowski as an open problem.
In this article we prove $\Pi_1^0$-hardness of this problem (the upper $\Pi_1^0$ bound is inherited by conservativity).

The rest of this article is organised as follows. In~\ref{S:Lrestr} we discuss Lambek's restriction and how to prove the desired complexity result
without this restriction imposed. In~\ref{S:grammar} we survey 
results on connections between Lambek grammars and
context-free ones.  The main part of the article is contained in~\ref{S:complex} and~\ref{S:Safiullin}. In~\ref{S:complex}, we represent
complexity results by Buszkowski and Palka for $\ACTomega$ and show how to strengthen the lower bound and prove $\Pi_1^0$-completeness
for the system without $\vee$ and $\wedge$. This complexity proof uses the version of Safiullin's theorem for the Lambek calculus without
Lambek's restriction. We prove it in~\ref{S:Safiullin}, which is the most technically hard section. Finally,~\ref{S:compl} and~\ref{S:conclusion}
contain some final remarks and discussions.

\section{Lambek's Restriction}\label{S:Lrestr}

A preliminary version of this article was presented at WoLLIC 2017 and published in its lecture notes~\citep{Kuznetsov2017WoLLIC}.
In the WoLLIC paper, we show $\Pi_1^0$-completeness of a system closely related to $\Lomega$---namely, the logic of residuated {\em semigroups} 
with positive iteration, denoted by $\Lpomega$.
From the logical point of view, in $\Lpomega$ there is no unit constant, and there should be always something on the
left-hand side. 

This constraint is called {\em Lambek's restriction}.  In the presence of Lambek's restriction one cannot add Kleene star, and it 
gets replaced by {\em positive iteration,}
or ``Kleene plus,'' with the following rules:
$$
\infer[({}^+\to)_\omega]{\Gamma, A^+, \Delta \to C}{\bigl(\Gamma, A^n, \Delta \to C\bigr)_{n = 1}^\infty}
\qquad
\infer[(\to {}^+)_n, n \geq 1]
{\Pi_1, \ldots, \Pi_n \to A^+}
{\Pi_1 \to A & \ldots & \Pi_n \to A}
$$
In the setting without Lambek's restriction Kleene plus is also available, being expressible in terms of Kleene star:
$A^+ = A \cdot A^*$.

Lambek's restriction was imposed on the calculus $\mathbf{L}$ 
in Lambek's original paper~\citep{Lambek1958} and is motivated linguistically~\citep[Sect.~2.5]{MootRetore}.
Unfortunately, there are no conservativity relations
between $\Lpomega$ and $\Lomega$. For example, $(p \BS p) \BS q \to q$ is derivable in $\LL$ (thus in $\Lomega$), but not in $\mathbf{L}$ (thus not in $\Lpomega$),
though the antecedent here is not empty. Hence, $\Pi_1^0$-hardness of the latter is not obtained automatically as a corollary.

Moreover, the proof of $\Pi_1^0$-hardness of $\Lpomega$ crucially depends on the following result by~\citet{Safiullin2007}:
any context-free language without the empty word can be generated by a Lambek grammar with unique type assignment, and Safiullin's proof
essentially uses Lambek's restriction. In this article, we feature a new result, namely, $\Pi_1^0$-completeness of $\Lomega$ itself.
We modify Safiullin's construction and extend his result to $\LL$ (which is already
interesting on its own). Next, we use this
new result in order to prove $\Pi_1^0$-hardness of $\Lomega$.

\section{Lambek Grammars and Context-Free Grammars}\label{S:grammar}

In this section we introduce Lambek categorial grammars and formulate equivalence results connecting them with
a more widely known formalism, context-free grammars. 
The notion of {\em Lambek grammar} is defined as follows:
\begin{definition}
A Lambek grammar over an alphabet $\Sigma$ is a triple $\Gc = \langle \Sigma, H, \rhd \rangle$, where
$H$ is a designated Lambek formula, called goal type, and $\rhd$ is a finite binary correspondence between
letters of $\Sigma$ and Lambek formulae. 
\end{definition}

\begin{definition}
A word $a_1 \dots a_n$ over $\Sigma$ is accepted by a grammar $\Gc$, if there exist formulae $A_1$, \ldots, $A_n$ such that
$a_i \rhd A_i$ ($i = 1, \ldots, n$) and the sequent $A_1, \ldots, A_n \to H$ is derivable in the Lambek calculus.
\end{definition}

\begin{definition}
The language generated by grammar $\Gc$ is the set of all words accepted by this grammar.
\end{definition}

In view of the previous section, one should distinguish grammars with and without Lambek's restriction:
the same grammar could generate different languages, depending on whether Lambek's restriction is imposed or not.

We also recall the more well-known notion of {\em context-free grammars.}

\begin{definition}
A context-free grammar over alphabet $\Sigma$ is a quadruple $\Gc = \langle \Nc, \Sigma, P, S \rangle$, where $\Nc$ is an auxiliary alphabet
of non-terminal symbols, not intersecting with $\Sigma$, 
 $S$ is a designated non-terminal symbol called the starting symbol, and $P \subset \Nc \times (\Nc \cup \Sigma)^*$ is
a finite set of production rules. Production rules are written in the form $A \Rightarrow \alpha$, where $A$ is a non-terminal symbol and
$\alpha$ is a word (possibly empty) over alphabet $\Nc \cup \Sigma$. Production rules with an empty $\alpha$ (of the form $A \Rightarrow \varepsilon$)%
\footnote{We use $\Lambda$ for the empty sequence of formulae and $\varepsilon$ for the empty word over an alphabet.} are
called $\varepsilon$-rules.
\end{definition}

\begin{definition}
A word $\eta\alpha\theta$ over $\Nc \cup \Sigma$ is immediately derivable from $\eta A \theta$ in context-free grammar $\Gc$
(notation: $\eta A \theta \Rightarrow_\Gc \eta \alpha\theta$), if  $(A \Rightarrow \alpha) \in P$. The  derivability relation
$\Rightarrow^*_\Gc$ is the reflexive-transitive closure of $\Rightarrow_\Gc$.
\end{definition}

\begin{definition}
The language generated by context-free grammar $\Gc$ is the set of all words over $\Sigma$ (i.e., without non-terminals) which are derivable from 
the starting symbol $S$: $\{ w \in \Sigma^* \mid S \Rightarrow_\Gc^* w \}$. Such languages are called context-free.
\end{definition}

Two grammars (for example, a Lambek grammar and a context-free one) are called {\em equivalent,} if they generate the same language.

Buszkowski's proof of $\Pi_1^0$-hardness of $\ACTomega$ uses the following translation of context-free grammars into Lambek grammars:
\begin{thm}[C. Gaifman, W. Buszkowski]\label{Th:Gaifman}
Any context-free grammar without $\varepsilon$-rules can be algorithmically transformed into an equivalent 
Lambek grammar, no matter with or without Lambek's restriction.
\end{thm}

This theorem was proved by Gaifman~\citep{BGS1960}, but with basic categorial grammars instead of Lambek grammars.
\citet{Buszkowski1985equivalence} noticed that Gaifman's construction works for Lambek grammars also.

The reverse translation is also available (in this article we do not need it):
\begin{thm}[M. Pentus]
Any Lambek grammar, no matter with or without Lambek's restriction,
 can be algorithmically transformed into an equivalent context-free grammar.~{\rm\citep{PentusCF}}
\end{thm}

For our purposes we shall need a refined version of Theorem~\ref{Th:Gaifman}.
\begin{definition}
A Lambek grammar is a grammar with unique type assignment, if for any $a \in \Sigma$ there exists
exactly one formula $A$ such that $a \rhd A$.
\end{definition}

\begin{thm}[A. Safiullin]\label{Th:Safiullin}
Any context-free grammar without $\varepsilon$-rules can be algorithmically transformed into an equivalent
Lambek grammar with unique type assignment with Lambek's restriction.~{\rm\citep{Safiullin2007}}
\end{thm}

Notice that Theorem~\ref{Th:Safiullin}, as formulated and proved by Safiullin, needs Lambek's restriction.
If one applied Safiullin's transformation and then abolished Lambek's restriction, the resulting grammar
could generate a different language, and the new grammar would not be equivalent to the original context-free one.

\section{Complexity of the Infinitary Calculi with and without Additives}\label{S:complex}
We start with the known results on algorithmic complexity of $\ACTomega$, infinitary action logic with additive connectives.

\begin{thm}[W. Buszkowski and E. Palka]\label{Th:BuszkoPalka} 
The derivability problem in $\ACTomega$ is $\Pi_1^0$-complete.~{\rm\citep{Buszkowski2007,Palka2007}}
\end{thm}

An algorithmic decision problem, presented as a set $\mathcal{A}$, belongs to $\Pi_1^0$, if there exists a decidable set 
$\mathcal{R}$ of pairs $\langle x,y \rangle$, such that $x \in \mathcal{A}$ if and only if $\langle x,y \rangle \in \mathcal{R}$ for all $y$.
The $\Pi_1^0$ complexity class is dual to $\Sigma_1^0$, the class of recursively enumerable sets\footnote{We suppose that sequents
of $\ACTomega$, as well as context-free grammars, are encoded as words over a fixed finite alphabet, or as natural numbers.}: a set is $\Pi_1^0$ if and only if
its {\em complement} is recursively enumerable.
The ``most complex'' sets in $\Pi_1^0$ are called $\Pi_1^0$-complete sets:
a set $\mathcal{A}$ is $\Pi_1^0$-complete, if (1) it belongs to $\Pi_1^0$ (upper bound); (2) it is {\em $\Pi_1^0$-hard,} that is, any other $\Pi_1^0$ set $\mathcal{B}$
is m-reducible to $\mathcal{A}$ (lower bound). The latter means that there exists a computable function $f$ such that $x \in \mathcal{A}$ if{f} $f(x) \in \mathcal{B}$.
By duality, a set is $\Pi_1^0$-complete if{f} its complement is $\Sigma_1^0$-complete. For example, since the halting problem for Turing machines is 
$\Sigma_1^0$-complete, the {\em non-halting} problem is $\Pi_1^0$-complete. 

Notice that a $\Pi_1^0$-complete set cannot belong to $\Sigma_1^0$: otherwise it would be decidable by Post's theorem, and this would lead to decidability of
all sets in $\Pi_1^0$, which is not the case. Thus, Theorem~\ref{Th:BuszkoPalka} implies the fact that the set of sequents provable in $\ACTomega$ is 
not recursively enumerable, and $\ACTomega$ itself cannot be reformulated as a system with finite derivations.

Division operations ($\SL$ and $\BS$) are essential for the $\Pi_1^0$ lower complexity bound. For the fragment $\ACTomega(\cdot,\vee,\KStar)$ a famous 
result by~\citet{Kozen1994IC} provides completeness of an inductive axiomatization for Kleene star and establishes PSPACE complexity. As we prove in this article,
however, the fragment with divisions and without additives, $\ACTomega(\BS,\SL,\cdot,\KStar)$, is still $\Pi_1^0$-complete.

The upper bound in Theorem~\ref{Th:BuszkoPalka} was proved by~\citet{Palka2007} using the following *-elimination technique.
For each sequent define its $n$-th approximation by replacing all negative occurrences of 
$A^*$ by $A^{\leq n} = \One \vee A \vee A^2 \vee \ldots \vee A^n$. Formally this is done by the following mutually recursive definitions:
\begin{align*}
& P_n(p_i) = p_i && N_n(p_i) = p_i \\
& P_n(\One) = \One && N_n(\One) = \One\\
& P_n(A \BS B) = N_n(A) \BS P_n(B) && N_n(A \BS B) = P_n(A) \BS N_n(B)\\
& P_n(B \SL A) = P_n(B) \SL N_n(A) && N_n(B \SL A) = N_n(B) \SL P_n(A)\\
& P_n(A \cdot B) = P_n(A) \cdot P_n(B) && N_n(A \cdot B) = N_n(A) \cdot N_n(B) \\
& P_n(A \vee B) = P_n(A) \vee P_n(B) && N_n(A \vee B) = N_n(A) \vee N_n(B) \\
& P_n(A \wedge B) = P_n(A) \wedge  P_n(B) && N_n(A \wedge B) = N_n(A) \wedge N_n(B) \\
& P_n(A^*) = \bigl(P_n(A)\bigr)^* && N_n(A^*) = \bigl(N_n(A)\bigr)^{\leq n}
\end{align*} 
Now the $n$-th approximation of a sequent $A_1, \ldots, A_k \to B$ can be defined as $N_n(A_1), \ldots, N_n(A_n) \to P_n(B)$.
\citet{Palka2007} proves that a sequent is derivable in $\ACTomega$ if and only if all its approximations are derivable. In a cut-free derivation of
an approximation, however, the $(\KStar\to)_\omega$ rule could never be applied, since there are no more negative occurrences of $\KStar$-formulae.
Proof search without the $\omega$-rule is decidable. Thus, we get $\Pi_1^0$ upper bound for $\ACTomega$.
By conservativity, this upper bound is valid for all elementary fragments of $\ACTomega$, in particular, for $\Lomega$.

For the lower bound ($\Pi_1^0$-hardness), \citet{Buszkowski2007} presents a reduction of a well-known
$\Pi_1^0$-complete problem, the totality problem for context-free grammars, to derivability in~$\ACTomega$.
Buszkowski's reduction is as follows. Let {\sc total}$^+$ denote the following algorithmic problem:
given a context-free grammar without $\varepsilon$-rules, determine whether it generates all non-empty words over $\Sigma$.
It is widely known that {\sc total}$^+$ is $\Pi_1^0$-complete~\citep{Sipser,DuKo}. The reduction of this problem to 
derivability in~$\ACTomega$ is performed as follows: given a context-free grammar, 
transform it (by Theorem~\ref{Th:Gaifman}) into an equivalent Lambek grammar $\Gc = \langle \Sigma, \rhd, H \rangle$. Suppose that
$\Sigma = \{ a_1, \ldots, a_m \}$ and for each $i = 1, \ldots, m$ the formulae that 
are in $\rhd$ correspondence with $a_i$ are $A_{i,1}, \ldots, A_{i,{n_i}}$.
Let 
$A_i = A_{i,1} \wedge \ldots \wedge A_{i,{n_i}}$ ($i = 1, \ldots, n$) and
$E = A_1 \vee \ldots \vee A_m$. Then the grammar generates all non-empty words if and only if
$E^+ \to H$ is derivable in $\ACTomega$~\citep[Lm.~5]{Buszkowski2007}. 
Recall that $E^+$ is $E^*, E$.
Thus, we have a reduction of {\sc total}$^+$ to the derivability problem in $\ACTomega$, and therefore the latter is $\Pi_1^0$-hard.

This construction essentially uses additive connectives, $\vee$ and $\wedge$. As shown by \citet{Buszkowski2007}, one can easily
get rid of $\vee$ by the following trick. First, notice that {\sc total}$^+$ is already $\Pi_1^0$-hard if we consider only languages
over a two-letter alphabet $\{a_1,a_2\}$~\citep{DuKo}. Second, $E^* = (A_1 \vee A_2)^*$ can be equivalently replaced
by $(A_1^* \cdot A_2)^* \cdot A_1^*$ (this equivalence is a law of Kleene algebra, thus provable in $\ACTomega$). The sequent $E^*,E \to H$
can be now replaced by a sequent without $\vee$ by the following chain of equivalent sequents:
\begin{align*}
& (A_1 \vee A_2)^*, A_1 \vee A_2 \to H \\
& (A_1^* \cdot A_2)^* \cdot A_1^*, A_1 \vee A_2 \to H \\
& (A_1^* \cdot A_2)^* \cdot A_1^* \to H \SL (A_1 \vee A_2)\\
& (A_1^* \cdot A_2)^* \cdot A_1^* \to (H \SL A_1) \wedge (H \SL A_2)
\end{align*}
(The last step is due to the equivalence $A \SL (B \vee C) \leftrightarrow (A \SL B) \wedge (A \SL C)$, which is provable
in $\mathbf{MALC}$ and thus in $\ACTomega$.)

Buszkowski also shows how to prove $\Pi_1^0$-hardness for $\ACTomega$ without $\wedge$---but then $\vee$ becomes irremovable.
We show how to get rid of both $\vee$ and $\wedge$ at once.

\begin{thm}\label{Th:main}
The derivability problem in $\Lomega$ is $\Pi_1^0$-complete.
\end{thm}

The upper bound follows from Palka's result by conservativity. 

For the lower bound ($\Pi_1^0$-hardness) we use the following {\em alternation}
problem, denoted by {\sc alt}$_2$, instead of {\sc total}$^+$. A context-free grammar over a two-letter alphabet $\{a_1, a_2\}$ belongs to
{\sc alt}$_2$, if the language it generates  includes all words beginning with $a_1$ and ending with $a_2$. (Other words can also belong to
this language.) The alternation problem {\sc alt}$_2$ is also $\Pi_1^0$-hard, by the following reduction from {\sc total}$^+$: take a context-free
grammar with starting symbol $S$ and append a new starting symbol $S'$ with rules $S' \Rightarrow a_1 S a_2$ and $S' \Rightarrow a_1 a_2$; 
the new grammar belongs to {\sc alt}$_2$ if and only if the original grammar belongs to {\sc total}$^+$.

Now we translate our context-free grammar to a Lambek grammar, as Busz\-kow\-ski does. It is easy to see 
that the grammar belongs to {\sc alt}$_2$ if and only if
the sequent $(A_1^+ \cdot A_2^+)^+ \to H$ is derivable in $\ACTomega$:
\begin{lem}\label{Lm:externalsimple}
The sequent $(A_1^+ \cdot A_2^+)^+ \to H$ is derivable in $\ACTomega$ if and only if so are
all sequents $A_1^{n_1}, A_2^{m_1}, A_1^{n_2}, A_2^{m_2}, \ldots, A_1^{n_k}, A_2^{n_k} \to H$ for
any $k$, $n_1$, $m_1$, $n_2$, $m_2$, \ldots, $n_k$, $m_k \ge 1$.
\end{lem} 
\begin{proof}
By cut, one can easily establish that the $\omega$-rule is invertible:
$$
\infer[(\CUT)]{\Gamma, A^n, \Delta \to C}{\infer[(\to\KStar)_n]{A^n \to A^*}{A \to A & \ldots & A \to A} & \Gamma, A^*, \Delta \to C}
$$
and so is $(\cdot\to)$:
$$
\infer[(\CUT)]{\Gamma, A, B, \Delta \to C}{\infer[(\to\cdot)]{A, B \to A \cdot B}{A \to A & B \to B} & \Gamma, A \cdot B, \Delta \to C}
$$
Now the ``if'' part goes by direct application of $(\cdot\to)$ and $(\KStar\to)_\omega$ and the ``only if'' one by their inversion.
\end{proof}

Now our proof of Theorem~\ref{Th:main} will be finished if we manage to formulate $A_1$ and $A_2$ without
$\wedge$. Recall that $A_i = A_{i,1} \wedge \ldots \wedge A_{i,{n_i}}$, where $A_{i,1}, \ldots, A_{i,{n_i}}$ are the formulae which are in the $\rhd$
correspondence with $a_i$. For a Lambek grammar with unique type assignment we have $n_i = 1$, and $A_i$ does not contain $\wedge$. Thus, Theorem~\ref{Th:main} now follows
from the fact that any context-free grammar without $\varepsilon$-rules can be equivalently transformed to a Lambek grammar with unique type assignment, without
Lambek's restriction. Indeed, our language belongs to {\sc alt}$_2$ if and only if any word of the form $a_1^{n_1} a_2^{m_1} \ldots a_1^{n_k} a_2^{n_k}$, for arbitrary
$k$, $n_1$, $m_1$, $n_2$, $m_2$, \ldots, $n_k$, $m_k \ge 1$, belong to the language. By definition of Lambek grammar with unique type assignment, 
this happens exactly when all sequents $A_1^{n_1}, A_2^{m_1}, A_1^{n_2}, A_2^{m_2}, \ldots, A_1^{n_k}, A_2^{n_k} \to H$ are derivable in $\LL$, and, by conservativity,
in $\ACTomega$.

In other words, we need an $\LL$-variant of Safiullin's Theorem~\ref{Th:Safiullin}, which we are going to prove in the next section
(Theorem~\ref{Th:SafiullinE}).

\section{Safiullin's Construction for $\LL$}\label{S:Safiullin}
In this section, we consider only the pro\-duct-free fragment $\LL(\BS,\SL)$. By ``\,$\vdash \Pi \to B$\,'' we mean that
$\Pi \to B$ is derivable in this calculus. By an $\LL(\BS,\SL)$-grammar we mean a Lambek grammar without Lambek's
restriction, where all formulae do not include the multiplication operation.

\begin{thm}\label{Th:SafiullinE}
Any context-free grammar without $\varepsilon$-rules can be algorithmically transformed to an equivalent
$\LL(\BS,\SL)$-grammar with
unique type assignment.
\end{thm}

Before presenting the construction of the grammar itself, we introduce
Safiullin's technique of proof analysis for the product-free Lambek calculus, adapted for the case without
Lambek's restriction. This technique has something in common with proof 
nets and focusing 
 for non-commutative linear logic; however, due to the simplicity of $\LL(\BS,\SL)$, this technique
works directly with Gentzen-style cut-free derivations.

We start with a simple and well-known fact:

\begin{lem}\label{Lm:reverse}
The $(\to\BS)$ and $(\to\SL)$ rules are reversible, i.e., if $\vdash \Pi \to A \BS B$, then $\vdash A,\Pi \to B$,
and if $\vdash \Pi \to B \SL A$, then $\vdash \Pi,A \to B$.
\end{lem}

\begin{proof}
 Derivability is established using cut (and then, if we want a cut-free derivation, eliminate cut).
 $$
 \infer[(\CUT)]{A, \Pi \to B}{\Pi \to A \BS B & \infer[(\BS\to)]{A, A \BS B \to B}{A \to A & B \to B}}
 \qquad
 \infer[(\CUT)]{\Pi, A \to B}{\Pi \to B \SL A & \infer[(\SL\to)]{B \SL A, A \to B}{A \to A & B \to B}}
 $$
\end{proof}

\begin{definition}
For any $\LL(\BS,\SL)$-formula let its {\em top} be a variable occurrence defined recursively as follows:
\begin{enumerate}
\item the top of a variable is this variable occurrence itself;
\item the top of $(A \BS B)$ and, symmetrically, of $(B \SL A)$ is the top of $B$.
\end{enumerate}
\end{definition}

For convenience we consider cut-free derivations with axioms of the form $q \to q$, where $q$ is a variable.
Axioms $A \to A$ with complex formulae $A$ are derivable (induction on $A$).

\begin{definition}
In a cut-free derivation of $\Pi \to q$, the {\em principal occurrence}\footnote{In proof nets, the principal occurrence
is the one connected by an axiom link to the $q$ in the succedent.} of $q$ in $\Pi$ is the one that
comes from the same axiom as the $q$ in the succedent. We denote the principal occurrence by
$\underline{q}$.
\end{definition}

Notice that the notion of principal occurrence depends on a concrete derivation, not just on the fact
of derivability of $\Pi \to q$.

\begin{lem}\label{Lm:top}
The principal occurrence is always a top.
\end{lem}

\begin{proof} Induction on derivation. \end{proof}

Using this lemma, one can locate possible principal occurrences by searching for tops with the same variable as the succedent (which has been reduced to
its top variable by Lemma~\ref{Lm:reverse}).

Introduce the following shortcut for {\em curried} passing of denominators: if $\Gamma = X_1, \ldots, X_n$ and $\Delta = Y_1, \dots, Y_m$ are
sequences of formulae, then let
$$
\Gamma \BS q \SL \Delta = X_n \BS \ldots \BS X_1 \BS q \SL Y_m \SL \ldots \SL Y_1.
$$
($\SL$ associates to the left and $\BS$ associates to the right).
If $\Gamma$ or $\Delta$ is empty, we just omit the corresponding divisions:
$$
\Lambda \BS q \SL \Delta = q \SL Y_m \SL \ldots \SL Y_1;
$$
$$
\Gamma \BS q \SL \Lambda = X_n \BS \ldots \BS X_1 \BS q;
$$
$$
\Lambda \BS q \SL \Lambda = q.
$$
(Since $(A \BS B) \SL C$ and $A \BS (B \SL C)$ are equivalent, we do not need parentheses here.)

Next follows the decomposition lemma, which allows reverse-engineering of $(\SL\to)$ and $(\BS\to)$, once we have
located the principal occurrence.

\begin{lem}\label{Lm:decomposition}
If $\vdash \Phi, (X_1, \ldots, X_n) \BS \underline{q} \SL (Y_1, \dots, Y_m), \Psi \to q$, then
$\Phi = \Phi_1, \dots, \Phi_n$ (some of $\Phi_i$ may be empty), $\vdash \Phi_i \to X_i$ for any $i = 1, \ldots, n$;
$\Psi = \Psi_1, \dots, \Psi_m$ (some of $\Psi_j$ may be empty), $\vdash \Psi_j \to Y_j$ for any $j = 1, \ldots, m$.
\end{lem}

\begin{proof}
Induction on derivation. By definition of principal occurrence, $\underline{q}$ always goes to the right branch in
applications of $(\SL\to)$ and $(\BS\to)$.
\end{proof}

In particular, Lemma~\ref{Lm:decomposition} yields the following corollary: if the principal occurrence is the top of a formula
of the form $p \SL \Gamma$, then it should be the leftmost formula in the antecedent.

The next ingredient of Safiullin's construction is the {\bf sentinel} formula. This formula is used
to delimit parts of the sequent and force them to behave independently. Safiullin uses a very simple
formula, $p \SL p$, as a sentinel. Without Lambek's restriction, however, this will not work, because
$\Lambda \to p \SL p$ is now derivable, making the sentinel practically useless. We use a more
complicated sentinel, using the technique of {\em raising}. A formula $A$ raised using $q$ is
$q \SL (A \BS q)$. Our sentinel is as follows: 
$$S_{p,q,r} = (r \SL (p \BS r)) \SL (q \SL (p \BS q)).$$
Variables $p$, $q$, and $r$ are parameters of the sentinel. We shall take fresh variables for them.

The top of $S_{p,q,r}$ is $r$. In the notation of Lemma~\ref{Lm:decomposition}, $S_{p,q,r} = r \SL (q \SL (p \BS q), p \BS r)$, thus, it is of the form $r \SL \Gamma$.

The following lemma trivialises the analysis of sequents of the form $S_{p,q,r}, \Pi \to S_{p,q,r}$, if the 
principal occurrence happens to be the top of the leftmost $S_{p,q,r}$.

\begin{lem}\label{Lm:sentinel_right}
If $\Pi$ does not have $p$ or $q$ in tops (in particular, $\Pi$ could include $S_{p,q,r}$, since the top of the sentinel is $r$) and $$\vdash \underline{r} \SL (q \SL (p \BS q), p \BS r), \Pi, q \SL (p \BS q), p \BS r \to r,$$ then 
$\Pi = \Lambda$.
\end{lem}

\begin{proof}
By Lemma~\ref{Lm:decomposition}, 
\begin{align*}
&\Pi, q \SL (p \BS q), p \BS r = \Psi_1, \Psi_2;\\
&\vdash \Psi_1 \to q \SL (p \BS q);\\
&\vdash \Psi_2 \to p \BS r.
\end{align*}
Inverting $(\to\SL)$ in the first sequent yields $\vdash \Psi_1, p \BS q \to q$.  If $\Psi_1$ does not include $q \SL (p \BS q)$, then the only
principal occurrence is the $q$ in $p \BS q$, and by Lemma~\ref{Lm:decomposition} $\vdash \Psi_1 \to p$, which contradicts Lemma~\ref{Lm:top}:
there are no $p$ tops in $\Psi_1$ (which is a subsequence of $\Pi$).

On the other hand, $\Psi_2 \ne \Lambda$, since $\not\vdash \Lambda \to p \BS r$. Thus, $\Psi_1 = \Pi, q \SL (p \BS q)$ and
$$
\vdash \Pi, \underline{q} \SL (p \BS q), p \BS q \to q \qquad\mbox{or} \qquad
\vdash \Pi, q \SL (p \BS q), p \BS \underline{q} \to q.
$$
In the first case, $\Pi = \Lambda$ by Lemma~\ref{Lm:decomposition}, {\em q.e.d.} In the second case,
$\vdash \Pi, q \SL (p \BS q) \to p$ fails, since there is no $p$ top in the antecedent.
\end{proof}

\begin{lem}\label{Lm:sentinel}
If $\Phi_0$, $\Phi_1$, \ldots, $\Phi_n$ do not have $p$, $q$, or $r$ in tops and $$\vdash \Phi_0, S_{p,q,r}, \Phi_1, S_{p,q,r}, \ldots,
\Phi_{n-1}, S_{p,q,r}, \Phi_n \to S_{p,q,r},$$
then $n = 1$ and all $\Phi_i$ are empty. (In other words, the only derivable sequent of this form is the trivial $S_{p,q,r} \to S_{p,q,r}$.)
In particular, $\not\vdash \Lambda \to S_{p,q,r}$ and $\not\vdash S_{p,q,r}, S_{p,q,r} \to S_{p,q,r}$.
\end{lem}

\begin{proof}
 Inverting
$(\to\SL)$ by Lemma~\ref{Lm:reverse}, we get
$$
\vdash \Phi_0, S_{p,q,r}, \Phi_1, S_{p,q,r}, \ldots,
\Phi_{n-1}, S_{p,q,r}, \Phi_n, q \SL (p \BS q), p \BS r \to r.
$$
Let us locate the principal occurrence of $r$. This should be a top, therefore it is either the rightmost $r$ in $p \BS r$ or
$r$ in one of the sentinels. In the first case we get, by Lemma~\ref{Lm:decomposition},
$$
\vdash \Phi_0, S_{p,q,r}, \Phi_1, S_{p,q,r}, \ldots,
\Phi_{n-1}, S_{p,q,r}, \Phi_n, q \SL (p \BS q) \to p,
$$
which immediately fails to be derivable, since there is no $p$ top in the antecedent to become the principal occurrence.

In the second case, since all sentinels are of the form $r \SL \Gamma$, the principal occurrence should be the leftmost one.
Thus, $\Phi_0 = \Lambda$ and the principal occurrence is the top of the leftmost $S_{p,q,r}$. (In particular, at least one sentinel should
exist, {\em i.e.,} $n \ge 1$.)
$$
\vdash \underline{r} \SL (q \SL (p \BS q), p \BS r), \Phi_1, S_{p,q,r}, \ldots,
\Phi_{n-1}, S_{p,q,r}, \Phi_n, q \SL (p \BS q), p \BS r \to r.
$$

Now by Lemma~\ref{Lm:sentinel_right} $\Phi_1, S_{p,q,r}, \ldots,
\Phi_{n-1}, S_{p,q,r}, \Phi_n$ should be empty. 
\end{proof}

Next, we are going to need the construction of {\bf joining formulae}~\citep{Pentus1994}. A joining formula for $A_1, \ldots, A_n$ is a formula
$B$ such that $A_i \to B$ is derivable for all $i = 1, \ldots, n$. The joining formula is a substitute for $A_1 \vee \ldots \vee A_n$ in the language
without $\vee$ and $\wedge$. One can also consider a joining formula for a set of sequences of formulae $\Gamma_1, \ldots, \Gamma_n$ as such $B$ that
$\Gamma_i \to B$ is derivable for all $i$.
Due to the substructural nature of the Lambek calculus, a joining formula does not always exist.
For example, two different variables, $p$ and $q$, are not joinable.  Pentus' criterion of joinability is based on the free group interpretation of the
Lambek calculus.

\begin{definition}
Let $\FG$ be the free group generated by $\Var$; $\varepsilon$ (the empty word) is its unit. 
Then $\nrm{A}$, the interpretation of Lambek formula $A$ in $\FG$, is defined recursively as follows:
\begin{itemize}
\item $\nrm{p} = p$ for $p \in \Var$;
\item $\nrm{\One} = \varepsilon$; 
\item $\nrm{A \cdot B} = \nrm{A} \nrm{B}$;
\item $\nrm{A \BS B} = \nrm{A}^{-1} \nrm{B}$;
\item $\nrm{B \SL A} = \nrm{B} \nrm{A}^{-1}$.
\end{itemize}
For a sequence $\Gamma$ of formulae its interpretation in $\FG$ is defined as follows:
\begin{itemize}
\item if $\Gamma = A_1, \ldots, A_n$, then $\nrm{\Gamma} = \nrm{A_1} \ldots \nrm{A_n}$;
\item $\nrm{\Lambda} = \varepsilon$.
\end{itemize}
\end{definition}

One can easily see (by induction on derivation) that if $A \to B$ is derivable, then $\nrm{A} = \nrm{B}$.\footnote{The free group interpretation can be
seen as a special case of the interpretation on residuated monoids, with the equality relation ($=$) taken as the preorder ($\preceq$). From this perspective,
the fact that derivability of $A \to B$ implies $\nrm{A} = \nrm{B}$ follows from the general soundness statement.} Thus,
if there exists a joining formula for a set of formulae (or sequences of formulae), then they should have the same free group interpretation.
In fact, this gives a criterion on the existence of a joining formula.
\begin{thm}\label{Th:join} 
If $\nrm{\Gamma_1} = \nrm{\Gamma_2} = \ldots = \nrm{\Gamma_n}$ then there exists a formula $B$ of the language of $\BS$, $\SL$, such that
$\vdash \Gamma_i \to B$ for all $i=1,\ldots, n$.
\end{thm}

This theorem is an easy corollary of the results of~\citet{Pentus1994}. Return to the calculus with multiplication, $\LL$. In this calculus, for each 
$\Gamma_i = A_{i,1}, \ldots, A_{i,n_i}$ consider its product $G_i  = A_{i,1} \cdot \ldots \cdot A_{i,{n_i}}$. If $\Gamma_i = \Lambda$, take $G_i = q \SL q$
for an arbitrary variable $q$. Sequents $\Gamma_i \to G_i$ are derivable. Then apply the main result from Pentus' article~\citep[Thm.~1]{Pentus1994} which
yields a joining formula $F$ for $\{ G_1, \ldots, G_n \}$. This formula could include the multiplication connective. However, for each formula $F$, possibly
with multiplication, there exists a formula $B$ in the language of $\BS$, $\SL$, such that $F \to B$ is derivable~\citep[Lm.~13(i)]{Pentus1994}.
By cut, we get $\vdash\Gamma_i \to B$.

A formula $A$ is called {\em zero-balanced} if $\nrm{A} = \varepsilon$. By Theorem~\ref{Th:join}, if all formulae in all $\Gamma_i$ are zero-balanced, then 
$\{ \Gamma_1, \ldots, \Gamma_n \}$ has a joining formula $B$.

\begin{lem}\label{Lm:sentinel_balance}
The sentinel formula $S_{p,q,r}$ is zero-balanced.
\end{lem}

\begin{proof}
Raising does not change the free group interpretation of a formula: $$\nrm{q \SL (A \BS q)} = 
\nrm{q} \nrm{A \BS q}^{-1} = q (\nrm{A}^{-1} \nrm{q})^{-1} = q q^{-1} \nrm{A} = \nrm{A}.$$
Now
$$
\nrm{S_{p,q,r}} = \nrm{r \SL (p \BS r)} \nrm{q \SL (p \BS q)}^{-1} = p p^{-1} = \varepsilon.
$$
\end{proof}

Now consider a set $\Uc = \{ A_1, \dots, A_n \}$ of zero-balance formulae and let $u, t, v, w, s$ be fresh variables, not occurring in $A_i$.
Consider the following two sets of sequences of formulae:
$$
\bigl\{ E_{i+} = S_{t,v,w}, A_i, S_{t,v,w}, A_{i+1}, S_{t,v,w}, \ldots, S_{t,v,w}, A_n, S_{t,v,w}  \mid i = 1, \ldots, n \bigr\};
$$
$$
\bigr\{ E_{i-} = S_{t,v,w}, A_1, S_{t,v,w}, A_2, S_{t,v,w}, \ldots, S_{t,v,w}, A_i, S_{t,v,w} \mid i = 1, \ldots, n \bigr\}.
$$
All formulae here are zero-balanced. Therefore, Theorem~\ref{Th:join} yields a joining formula for each set: there
exist formulae $F$ and $G$ (in the language of $\BS$ and $\SL$) such that for all $i = 1, \dots, n$ we have
$$
\vdash E_{i+}  \to F \qquad\mbox{and}\qquad \vdash E_{i-} \to G.
$$
Let
$$
E = S_{t,v,w}, A_1, S_{t,v,w}, A_2, S_{t,v,w}, \ldots, S_{t,v,w}, A_n, S_{t,v,w};
$$
$$
B = E, (((u \SL F) \BS u) \BS S_{t,v,w});
\qquad
C = (S_{t,v,w} \SL (u \SL (G \BS u))), E;
$$
$$
\is(\Uc) = (s \SL E, B) \BS s \SL C.
$$
The formula $\is(\Uc)$, in a sense, would play the r\^{o}le of $A_1 \vee \ldots \vee A_n$.

We also define versions of $E_{j+}$ and $E_{j-}$, for $j = 1, \ldots, n$, which lack the sentinel on one edge: 
$$
E'_{j+} = A_j, S_{t,v,w}, \ldots, A_n, S_{t,v,w};
$$
$$
E'_{j-} = S_{t,v,w}, A_1, \ldots, S_{t,v,w}, A_j.
$$
For convenience, we also define $E'_{0-} = E'_{(n+1)+} = \Lambda$. Now for any $i = 1, \ldots, n$
we have
$$
E = E_{i-}, E'_{(i+1)+} = E'_{(i-1)-}, E_{i+} = E'_{(i-1)-}, S_{t,v,w}, A_i, S_{t,v,w}, E'_{(i+1)+}.
$$

\begin{lem}\label{Lm:forward}
$\vdash A_i \to \is(\Uc)$ for any $A_i \in \Uc$.
\end{lem}

\begin{proof}
By $(\to\SL)$, $(\to\BS)$, $(\cdot\to)$, $A_i \to \is(\Uc)$ is derivable from $s \SL E, B, A_i, C, \to s$. 
By definition of $B,C$, the latter is the same as
$$
s \SL E, E'_{(i-1)-}, E_{i+}, ((u \SL F) \BS u) \BS S_{t,v,w}, A_i, S_{t,v,w} \SL (u \SL (G \BS u)), E_{i-}, E'_{(i+1)+} \to s.
$$
By construction, we have
$$
\vdash E_{i+} \to F \to (u \SL F) \BS u
\qquad\mbox{and}\qquad
\vdash E_{i-} \to G \to u \SL (G \BS u).
$$
Thus, applying $(\BS\to)$ and $(\SL\to)$, we reduce to
$$
s \SL E, E'_{(i-1)-}, S_{t,v,w}, A_i, S_{t,v,w}, E'_{(i+1)+} \to s.
$$
This sequent is exactly $s \SL E, E \to s$, which is derivable by several applications of $(\SL\to)$ (recall that $E$ is a sequence of formulae).
\end{proof}

\begin{lem}\label{Lm:back}
Let $\Pi$ be a non-empty sequence of types whose tops are not $s$, $t$, $v$, $w$, and let $\vdash \Pi \to \is(\Uc)$. Then for some
$A_j \in \Uc$ we have $\vdash B_2, \Pi, C_1 \to A_j$, where $B_2$ is a suffix of $B$ and $C_1$ is a prefix of $C$.
\end{lem}

\begin{proof}
Inverting $(\to\SL)$ and $(\to\BS)$ gives $\underline{s} \SL E, B, \Pi, C \to s$ ($B$, $\Pi$, and $C$ do not have $s$ in tops).
Now apply Lemma~\ref{Lm:decomposition}. The sequence $B, \Pi, C$ gets split into $\Psi_1, \ldots, \Psi_{2n+1}$, and we have
$\vdash \Psi_1 \to S_{t,v,w}$, $\vdash \Psi_2 \to A_1$, $\vdash \Psi_3 \to S_{t,v,w}$, \ldots, $\vdash \Psi_{2n} \to A_n$, $\vdash \Psi_{2n+1} \to S_{t,v,w}$.

Here we consider the interesting case of a non-empty $\Pi$. The case of $\Pi = \Lambda$ is similar and is handled in Lemma~\ref{Lm:back_empty} below.
If $\Pi$ as a whole comes into one
of $\Psi_{2j}$, $\vdash \Psi_{2j} \to A_j$, this is exactly what we want ($\Psi_{2j} = B_2, \Pi, C_1$).

In the other case, there are two possibilities: either 
a non-empty part of $\Pi$, denoted by $\Pi'$, comes to a part
$\Psi_{2j+1} \to S_{t,v,w}$, where $\Psi_{2j+1} = B_2, \Pi', C_1$ ($B_2$ is a suffix of $B$, $C_1$ is a prefix of $C$;
if $\Pi'$ is not the whole $\Pi$, one of them is empty), or 
for some $j$ we have $\Psi_{2j+1} = \Lambda$, and parts of $\Pi$ come to $\Psi_{2j}$ and $\Psi_{2j+2}$.
The latter case is impossible, since $\not\vdash \Lambda \to S_{t,v,w}$ (Lemma~\ref{Lm:sentinel})\footnote{This is the difference from
Safiullin's sentinel!}.
 In the former case we also wish to obtain
contradiction.

Inverting $(\to\SL)$ gives $\vdash B_2, \Pi', C_1, v \SL (t \BS v), t \BS w \to w$. 
We prove that such a sequent cannot be derivable, proceeding by induction on the number of formulae in $B_2$.
Suppose the sequent is derivable. Locate the principal occurrence of $w$. 
Since all tops in $\Pi'$ are not $w$, and so are all $A_j$, this 
principal occurrence is either the rightmost one, or located in $S_{t,v,w}$.
Moreover, formulae with $S_{t,v,w}$ are of the form $w \SL \Gamma$, except for 
the last formula in $B$, which is $((u \SL F) \BS u) \BS S_{t,v,w}$. Thus, 
the principal occurrence is either in the leftmost $S_{t,v,w}$ of $B_2$, or in $((u \SL F) \BS u) \BS S_{t,v,w}$ in the end of $B_2$.
The non-emptiness of $\Pi'$ prevents  it from being in the first formula of $C_1$, which is of the form $S_{t,v,w} \SL (u \SL (G \BS u))$.
Consider these three possible cases.

{\em Case 1:} $\vdash B_2, \Pi', C_1, v \SL (t \BS v), t \BS \underline{w} \to w$.
By Lemma~\ref{Lm:decomposition}, we obtain derivability of $B_2, \Pi', C_1, v \SL (t \BS v) \to t$, which immediately fails, since there are no $t$ tops in
the antecedent.

{\em Case 2:} the principal occurrence is in the leftmost $S_{t,v,w}$ in $B_2$:
$$
\vdash \underline{w} \SL (v \SL (t \BS v), t \BS w), B'_2, \Pi', C_1, v \SL (t \BS v), t \BS w \to w.
$$
By Lemma~\ref{Lm:sentinel_right}, $B'_2, \Pi', C_1 = \Lambda$, which contradicts the non-emptiness of $\Pi'$.

{\em Case 3:} the principal occurrence is in the last formula of $B_2$:
$$
\vdash B''_2, ((u \SL F) \BS u) \BS \underline{w} \SL (v \SL (t \BS v), t \BS w), \Pi', C_1, v \SL (t \BS v), t \BS w \to w.
$$
In this case essentially the same happens: by Lemma~\ref{Lm:decomposition} we have
\begin{align*}
& \vdash B''_2 \to (u \SL F) \BS u;\\
& \Pi', C_1, v \SL (t \BS v), t \BS w = \Psi_1, \Psi_2;\\
& \vdash \Psi_1 \to v \SL (t \BS v);\\
& \vdash \Psi_2 \to t \BS w.
\end{align*}
Again, as in the proof of Lemma~\ref{Lm:sentinel_right}, $\Psi_1$ should be $\Pi', C_1, v \SL (t \BS v)$, the principal occurrence
is located as follows: 
$\vdash \Pi', C_1, \underline{v} \SL (t \BS v), t \BS v \to v$, and $\Pi', C_1 = \Lambda$, which contradicts with $\Pi' \ne \Lambda$.
\end{proof}

\begin{lem}\label{Lm:back_empty}
If $\vdash \Lambda \to \is(\Uc)$, then we have $\vdash B_2, C_1 \to A_j$, where $B_2$ is a suffix of $B$ and $C_1$ is a prefix of $C$.
\end{lem}
Notice that $B_2$ and $C_1$ are allowed to be empty. In particular, we could get just $\vdash \Lambda \to A_j$.

\begin{proof}
As in the proof of the previous lemma, we get $B, C = \Psi_1, \ldots, \Psi_{2n+1}$, where $\vdash \Psi_{2j+1} \to S_{t,v,w}$ for $j = 0, 1, \ldots, n$ and
$\vdash \Psi_{2i} \to A_i$ for $i = 1, \ldots, n$. The only bad case is when {\em both} the last formula of $B$ and the first formula of $C$ come to $\Psi_{2j+1}$:
otherwise, for some $i$, $\Psi_{2i} = B_2, C_1$ ($B_2$, or $C_1$, or even both could be empty), which is what we need.

In the bad case, we have $\vdash B_2, C_1 \to S_{t,v,w}$, where both $B_2$ and $C_1$ are non-empty. Inverting $(\to\SL)$ yields
$$
\vdash B_2, C_1, v \SL (t \BS v), t \BS w \to w,
$$
and there are again three possibilities for the principal $w$, exactly as in the proof of the previous lemma. Notice that the usage of the first formula 
in $C_1$ for the principal occurrence is now prohibited by non-emptiness of $B_2$, not $\Pi'$.

{\em Case 1:} $\vdash B_2, C_1, v \SL (t \BS v), t \BS \underline{w} \to w$.  Exactly as Case~1 in the previous lemma.

{\em Case 2:} the principal occurrence is in the leftmost $S_{t,v,w}$ in $B_2$. As in Case~2 of the previous lemma,
we get $B_2', C_1 = \Lambda$, which contradicts non-emptiness of $C_1$.

{\em Case 3:} the principal occurrence is in the last formula of $B_2$:
$$\vdash B''_2, ((u \SL F) \BS u) \BS \underline{w} \SL (v \SL (t \BS v), t \BS w), C_1, v \SL (t \BS v), t \BS w \to w.$$
This boils out into $C_1 = \Lambda$ (see proof of the previous lemma), which is false.
\end{proof}

Now we come to the {\bf main part of the construction.} Consider a language without the empty word 
generated by a context-free
grammar $\Gc$ in Greibach normal form~\citep{Greibach1965}. Let $\Sigma = \{ a_1, \ldots, a_\mu \}$
be its terminal alphabet and $\Nc = \{ N_0, N_1, \ldots, N_\nu \}$ be the non-terminal one; $N_0$ is the
starting symbol. 

Production rules of $\Gc$ are of the form
$$
N_i \Rightarrow a_j N_k N_\ell, \qquad\mbox{or}\qquad
N_i \Rightarrow a_j N_k, \qquad\mbox{or}\qquad
N_i \Rightarrow a_j.
$$
In order to simplify the proof, we write all these rules uniformly:
$$
N_i \Rightarrow a_j N_k^? N_\ell^?,
$$
where ${}^{?}$ means optionality.

Our variables (all distinct) will be $x$, $z$, $p_i$, $q_i$ $r_i$ ($0 \leq i \leq \nu$), plus 
$u,t,v,w,s$ for the construction described above.
Introduce the following formulae:
\begin{center}
\begin{tabular}{ll}
$H_i = (z \SL z) \SL S_{p_i,q_i,r_i}$ & \quad for each $N_i \in \Nc$ ($0 \leq i \leq \nu$);\\[2pt]
$A_{j;i,k^?,\ell^?} = x \SL ((H_k^?, H_\ell^?, S_{p_i,q_i,r_i}) \BS x)$ & \quad for each production rule $N_i \Rightarrow a_j N_k^? N_\ell^?$.
\end{tabular}
\end{center}

\begin{lem}
${A_{j;i,k^?,\ell^?}}$ is a zero-balance formula.
\end{lem}

\begin{proof}
$\nrm{A_{j;i,k^?,\ell^?}} = x x^{-1} \nrm{H_k}^? \nrm{H_\ell}^? \nrm{S_{p_i,q_i,r_i}}$.

$\nrm{S_{p_i,q_i,r_i}} = \varepsilon$ by Lemma~\ref{Lm:sentinel_balance}.

$\nrm{H_k} = z z^{-1} \nrm{S_{p_k,q_k,r_k}}^{-1} = \varepsilon$.
\end{proof}

Let $\Uc_j$ be the set of all $A_{j;i,k^?,\ell^?}$ for each $a_j$ ($1 \leq j \leq \mu$). Since all formulae
in $\Uc_j$ are zero-balanced, we can construct $\is(\Uc_j)$ obeying Lemma~\ref{Lm:forward}, Lemma~\ref{Lm:back}, and
Lemma~\ref{Lm:back_empty}.
Finally, let $K_j = (z \SL z) \SL \is(\Uc_j)$.

In our derivation analysis, we shall frequently use Lemma~\ref{Lm:decomposition} (decomposition of division in the antecedent). This requires
locating the principal occurrence of a specific variable. Since the principal occurrence is always a top, we need to recall the tops of formulae
used in our construction. For the convenience of the reader, we gather all of them in one table:

\begin{center}
\begin{tabular}{l@{\ }|@{\ }l}
{\bf formula} & {\bf top variable} \\[3pt]
$A_{j;i,k^?,l^?}$ & $x$\\
$K_j$ & $z$\\
$H_i$ & $z$\\
$S_{p_i,q_i,r_i}$ & $r_i$\\
formulae of $B$ and $C$ & $x$ or $w$
\end{tabular}
\end{center}

Now we are ready to formulate the key lemma:
\begin{lem}\label{Lm:key}
For any non-empty word $a_{i_0} a_{i_1} \ldots a_{i_n}$ and any $N_m \in \Nc$, the word $a_{i_1} \ldots a_{i_n}$ is derivable from
$N_m$ in $\Gc$ if and only if $\vdash K_{i_0}, K_{i_1}, \ldots, K_{i_n} \to H_m$.
\end{lem}

Before proving it, we state and prove yet two technical statements.
\begin{lem}\label{Lm:badcase}
No sequent of the form $\widetilde{B}_2, K_{i_1}, \ldots, K_{i_m} \to \is(\Uc_j)$, where $\widetilde{B}_2$ is a suffix
of $B$ (maybe empty), is derivable. 
In particular, $\not\vdash K_{i_1}, \ldots, K_{i_m} \to \is(\Uc_j)$ and $\not\vdash \Lambda \to \is(\Uc_j)$.
\end{lem}

\begin{proof}
In order to make the notation shorter, let $\vec{K} = K_{i_1}, \ldots, K_{i_m}$
Suppose that the sequent in question is derivable and apply Lemma~\ref{Lm:back} or Lemma~\ref{Lm:back_empty}, depending on whether
$\widetilde{B}_2, \vec{K}$ is empty:
$$
\vdash B_2, \widetilde{B}_2, \vec{K}, C_1 \to x \SL ((H_k^?, H_\ell^?, S_{p_i,q_i,r_i}) \BS x)
$$
for some $A_{j;i,k^?,\ell^?} = x \SL ((H_k^?, H_\ell^?, S_{p_i,q_i,r_i}) \BS x) \in \Uc_j$.

Here $B_2$, or $C_1$, or both can be empty. In particular, Lemma~\ref{Lm:back_empty} could yield $\vdash \Lambda \to x \SL ((H_k^?, H_\ell^?, S_{p_i,q_i,r_i}) \BS x)$.
In  this case, however, we get $\vdash (H_k^?, H_\ell^?, S_{p_i,q_i,r_i}) \BS x \to x$, and by Lemma~\ref{Lm:decomposition} $\vdash \Lambda \to S_{p_i,q_i,r_i}$, which
is not true.

Let $\widehat{B}_2 = B_2, \widetilde{B}_2$. This sequence is not always a suffix of $B$; however, it keeps the property we shall need:
all formulae in $\widehat{B}_2$ have tops $x$ or $w$, and a formula with top $x$ is always of the form $x \SL \Gamma$ and
 could not be the last formula of $\widehat{B}_2$.

 Inverting $(\to\SL)$ (Lemma~\ref{Lm:reverse}) yields:
$$
\vdash \widehat{B}_2, \vec{K}, C_1, (H_k^?, H_\ell^?, S_{p_i,q_i,r_i}) \BS x \to x.
$$
Locate the principal occurrence of $x$. Recall that tops of all $K_i$'s are $z$, thus this top is either in one of the
$A$'s in $\widehat{B}_2$ or $C_1$, or the rightmost occurrence of $x$. Since all $A$'s are of the form $x \SL \Gamma$, the former is possible only if
it is the leftmost occurrence. Moreover, this leftmost occurrence should be in $\widehat{B}_2$, because, even if $\widehat{B}_2$ and $\vec{K}$ happen to be empty,
the leftmost formula of $C$ (and, thus, of its prefix $C_1$) is not one of the $A$'s.

Thus, we have to consider two cases.

{\em Case 1.} The principal $x$ is in the leftmost formula of the form $A_{j';i',{k'}^?,{\ell'}^?}$ in $\widehat{B}_2$:
$$
\vdash \underline{x} \SL ((H_{k'}^?, H_{\ell'}^?, S_{p_{i'},q_{i'},r_{i'}}) \BS x), \widehat{B}'_2, \vec{K}, C_1, (H_k^?, H_\ell^?, S_{p_i,q_i,r_i}) \BS x \to x.
$$
Notice that $A_{j';i',{k'}^?,{\ell'}^?}$, as a formula with top $x$, is not the rightmost formula of $\widehat{B}_2$, thus $\widehat{B}'_2 \ne \Lambda$. Lemma~\ref{Lm:decomposition} gives
$$
\vdash \widehat{B}'_2, \vec{K}, C_1, (H_k^?, H_\ell^?, S_{p_i,q_i,r_i}) \BS x \to (H_{k'}^?, H_{\ell'}^?, S_{p_{i'},q_{i'},r_{i'}}) \BS x.
$$
Inverting $(\to\BS)$ (Lemma~\ref{Lm:reverse}) yields
$$
\vdash H_{k'}^?, H_{\ell'}^?, S_{p_{i'},q_{i'},r_{i'}}, \widehat{B}'_2, \vec{K}, C_1, (H_k^?, H_\ell^?, S_{p_i,q_i,r_{i'}}) \BS x \to x.
$$
This situation looks similar to the one in the beginning of the proof, but now there is only one choice of the principal $x$.
Indeed, tops of $H_i$ and $K_i$ are not $x$. The $A$'s in $\widehat{B}'_2$ and $C_1$ also could not include the principal $x$, since in this
case it would be the leftmost one. This is not true, since there is at least $S_{p_{i'},q_{i'},r_{i'}}$ to the left. Thus,
$$
\vdash H_{k'}^?, H_{\ell'}^?, S_{p_{i'},q_{i'},r_{i'}}, \widehat{B}'_2, \vec{K}, C_1, (H_k^?, H_\ell^?, S_{p_i,q_i,r_i}) \BS \underline{x} \to x,
$$
and by Lemma~\ref{Lm:decomposition} we have
\begin{align*}
& H_{k'}^?, H_{\ell'}^?, S_{p_{i'},q_{i'},r_{i'}}, \widehat{B}'_2, \vec{K}, C_1 = \Psi_1, \Psi_2, \Psi_3;\\
& \vdash \Psi_1 \to H_k^?; \\
& \vdash \Psi_2 \to H_\ell^?; \\
& \vdash \Psi_3 \to S_{p_i,q_i,r_i}.
\end{align*}
Here and further we use the following convention about optionality: if $\vdash \Gamma \to H^?_k$, but there is no $H_k$, 
then this statement means $\Gamma = \Lambda$; same for $H_\ell$. We also impose priority of optionality: if there is $H_\ell$, there
is also $H_k$ (otherwise we rename $H_\ell$ to $H_k$).

Notice that tops of $H_i$, $K_i$, and formulae of $\widehat{B}'_2$ and $C_1$ are not $p_i$, $q_i$, or $r_i$. 
Thus, by Lemma~\ref{Lm:sentinel}, $i'$ should be equal to $i$ and $\Psi_3$ should be $S_{p_i,q_i,r_i}$. This yields $\widehat{B}'_2 = \Lambda$.
Contradiction:
we have shown above that $\widehat{B}'_2 \ne\Lambda$.

{\em Case 2.} The principal $x$ is the rightmost one:
$$
\vdash \widehat{B}_2, \vec{K}, C_1, (H_k^?, H_\ell^?, S_{p_i,q_i,r_i}) \BS \underline{x} \to x.
$$
Lemma~\ref{Lm:decomposition} yields
\begin{align*}
& \widehat{B}_2, \vec{K}, C_1 = \Psi_1, \Psi_2, \Psi_3;\\
& \vdash \Psi_1 \to H_k^?;\\
& \vdash \Psi_2 \to H_\ell^?;\\
& \vdash \Psi_3 \to S_{p_i,q_i,r_i}.
\end{align*}
The third derivability, $\vdash \Psi_3 \to S_{p_i,q_i,r_i}$, contradicts Lemma~\ref{Lm:sentinel}: there are no $p_i$, $q_i$, or $r_i$ tops in the
left-hand side.
\end{proof}

\begin{lem}\label{Lm:badcase2}
Let $\Phi$ be a non-empty sequence of formulae whose tops are not $p_i$, $q_i$, $r_i$; any formula with top $z$ in $\Phi$ is of the form $z \SL \Gamma$;
the top of the leftmost formula in $\Phi$ is not $z$. Then
no sequent of the form $H_{k'}^?, H_{\ell'}^?, S_{p_{i'},q_{i'},r_{i'}}, \Phi \to H_k$ is derivable.
\end{lem}

\begin{proof}
Suppose the contrary.
Recall that $H_k = (z \SL z) \SL S_{p_k,q_k,r_k} = z \SL (S_{p_k,q_k,r_k}, z)$ and apply Lemma~\ref{Lm:reverse}:
$$
\vdash H_{k'}^?, H_{\ell'}^?, S_{p_{i'},q_{i'},r_{i'}}, \Phi, S_{p_k,q_k,r_k}, z \to z.
$$
Since all $z$ tops come in formulae of the form $z \SL \Gamma$, the principal occurrence should be the leftmost one.
(The standalone rightmost $z$ could not be principal, because then by Lemma~\ref{Lm:decomposition} it should have been the only formula in the antecedent.)
 This could happen only
in $H_{k'}$ (thus, it should exist): the sentinel $S_{p_{i'},q_{i'},r_{i'}}$ blocks other possibilities. Thus, we have
$$
\vdash\underline{z} \SL (S_{p_{k'},q_{k'},r_{k'}}, z), H_{\ell'}^?, S_{p_{i'},q_{i'},r_{i'}}, \Phi, S_{p_k,q_k,r_k}, z \to z.
$$
Apply Lemma~\ref{Lm:decomposition}:
\begin{align*}
& H_{\ell'}^?, S_{p_{i'},q_{i'},r_{i'}}, \Phi, S_{p_k,q_k,r_k}, z = \Psi_1, \Psi_2;\\
& \vdash \Psi_1 \to S_{p_{k'},q_{k'},r_{k'}};\\
& \vdash \Psi_2 \to z.
\end{align*}
By Lemma~\ref{Lm:sentinel} $i' = k'$ and the first part, $\Psi_1$, should include only $S_{p_{i'},q_{i'},r_{i'}}$. 
Therefore,
there is no $H_{\ell'}$, and $\Psi_2 = \Phi, S_{p_k,q_k,r_k}, z$.

Now we have
$$
\vdash \Phi, S_{p_k,q_k,r_k}, z \to z
$$
and we fail to locate the principal occurrence of $z$. It could be only in $\Phi$, and since there $z$ tops appear only in formulae of the form $z \SL \Gamma$,
it should be the leftmost one. Contradiction: the leftmost formula of $\Phi$ does not have top $z$.
\end{proof}

\vskip 5pt
Now we finally {\bf prove the key lemma.}

\vskip 3pt
\begin{proof}[Proof of Lemma~\ref{Lm:key}] For convenience let $j = i_0$.

{The {\bf ``only if''} part} is easier. Proceed by induction on the derivation of $a_j a_{i_1} \dots a_{i_n}$ from $N_m$ in $\Gc$. Consider the first
production rule in this derivation: $N_m \Rightarrow a_j N^?_k N^?_\ell$. Here $N^?_k$ derives $a_{i_1} \dots a_{i_{n'}}$ and $N^?_\ell$ derives $a_{i_{n'+1}} \dots
a_{i_n}$. (If there is no $N_\ell$, then $n' = n$; if there are no $N_k$ and $N_\ell$, then $n=0$.) By induction hypothesis we have the following:
$$
\begin{aligned}
& \vdash K_{i_1}, \dots, K_{i_{n'}} \to H^?_k \\
& \vdash K_{i_{n'+1}}, \dots, K_{i_{n}} \to H^?_\ell
\end{aligned}
$$
Recall our convention about optionality: if $\vdash \Gamma \to H^?_k$, but there is no $H_k$, 
then this statement means $\Gamma = \Lambda$; same for $H_\ell$. 

Next, since the production rule $N_m \Rightarrow a_j N^?_k N^?_\ell$ is in $\Gc$, the formula $A_{j;m,k^?,\ell^?}$ belongs to $\Uc_j$, and by
Lemma~\ref{Lm:forward} we have
$$
\vdash A_{j;m,k^?,\ell^?} \to \is(\Uc_j)
$$
or, explicitly,
$$
\vdash x \SL ((H_k^?, H_\ell^?, S_{p_m,q_m,r_m}) \BS x) \to \is(\Uc_j).
$$
By cut with $H_k^?, H_\ell^?, S_{p_m,q_m,r_m} \to x \SL ((H_k^?, H_\ell^?, S_{p_m,q_m,r_m}) \BS x)$ we get
$$
\vdash H_k^?, H_\ell^?, S_{p_m,q_m,r_m} \to \is(\Uc_j).
$$
Now the necessary sequent $K_j, K_{i_1}, \ldots, K_{i_{n'}}, K_{i_{n'+1}}, \ldots, K_{i_n} \to H_m$ is derived as 
follows. Cut of $K_{i_1}, \ldots, K_{i_{n'}} \to H_k^?$, $K_{i_{n'+1}}, \ldots, K_{i_n} \to H_\ell^?$, and
$H_k^?, H_\ell^?, S_{p_m,q_m,r_m} \to \is(\Uc_j)$ gives $K_{i_1}, \ldots, K_{i_n}, S_{p_m,q_m,r_m} \to \is(\Uc_j)$
(if there is no $H_k$ and/or $H_\ell$, we omit the corresponding premise). Next, by $(\SL\to)$ and $(\to\SL)$ we
get $$\vdash (z \SL z) \SL \is(\Uc_j), K_{i_1}, \ldots, K_{i_{n'}}, K_{i_{n'+1}}, \ldots, K_{i_n} \to (z \SL z) \SL S_{p_m,q_m,r_m}.$$
This is the necessary sequent, since $K_j = (z \SL z) \SL \is(\Uc_j)$ and $H_m = (z \SL z) \SL S_{p_m,q_m,r_m}$.

For the  {\bf ``if''} part, proceed by induction on a cut-free derivation of the sequent $K_j, K_{i_1}, \ldots, K_{i_n} \to H_m$. 
Basically, we show that the only way this derivation could go is the one shown in the ``only if'' part of this proof.

Recall that
$H_m = (z \SL z) \SL S_{p_m,q_m,r_m} = z \SL (S_{p_m,q_m,r_m}, z)$ and reverse $(\to\SL)$ twice (Lemma~\ref{Lm:reverse}). Thus
we get 
$$
\vdash K_j, K_{i_1}, \ldots, K_{i_n}, S_{p_m,q_m,r_m}, z \to z.
$$
Each $K_i$ is of the form $z \SL \Gamma$. 
Thus, the principal occurrence of $z$ is the top of the leftmost $K_i$, that is, $K_j = (\underline{z} \SL z) \SL \is(\Uc_j) = 
\underline{z} \SL (\is(\Uc_j), z)$, and we have 
$$\begin{aligned}
& K_{i_1}, \ldots, K_{i_n}, S_{p_m,q_m,r_m}, z = \Psi_1, \Psi_2;\\
& \vdash \Psi_1 \to \is(\Uc_j); \\ 
& \vdash \Psi_2 \to z.
\end{aligned}
$$

If $\Psi_1$ contains only $K$'s, then $\Psi_1 \to \is(\Uc_j)$ is not derivable by Lemma~\ref{Lm:badcase}. 
Moreover, $\Psi_2 \ne \Lambda$, since $\not\vdash \Lambda \to z$. Thus,
$\Psi_1 = K_{i_1}, \ldots, K_{i_n}, S_{p_m,q_m,r_m}$ and
$$
\vdash K_{i_1}, \ldots, K_{i_n}, S_{p_m, q_m,r_m} \to \is(\Uc_j).
$$
Let $\vec{K} = K_{i_1}, \ldots, K_{i_n}$. 
By Lemma~\ref{Lm:back} (recall that $\vec{K}, S_{p_m,q_m,r_m}$ is definitely non-empty),
$$
\vdash B_2, \vec{K}, S_{p_m,q_m,r_m}, C_1 \to A_{j;i,k^?,\ell^?}
$$
for some $A_{j;i,k^?,\ell^?} \in \Uc_j$. Since $A_{j;i,k^?,\ell^?} = x \SL ((H_k^?,H_\ell^?,S_{p_i,q_i,r_i}) \BS x)$, inversion of $(\to\BS)$ yields
$$
\vdash B_2, \vec{K}, S_{p_m,q_m,r_m}, C_1, (H_k^?,H_\ell^?,S_{p_i,q_i,r_i}) \BS x \to x.
$$
Locate the principal occurrence of $x$. It is either in the leftmost formula in $B_2$, of the form $A_{j';i',{k'}^?,{\ell'}^?}$, or the rightmost occurrence of $x$.

{\em Case 1.}
$$
\vdash \underline{x} \SL ((H_{k'}^?, H_{\ell'}^?, S_{p_{i'},q_{i'},r_{i'}}) \BS x), 
B'_2, \vec{K}, S_{p_m,q_m,r_m}, C_1, (H_k^?,H_\ell^?,S_{p_i,q_i,r_i}) \BS x \to x.
$$
The important notice here is that $B'_2$ {\em is not empty,} since $A_{j';i',{k'}^?,{\ell'}^?}$ was not the rightmost formula in $B_2$.

Decomposition (Lemma~\ref{Lm:decomposition}) and inversion of $(\to\BS)$ (Lemma~\ref{Lm:reverse}) yields
$$
\vdash H_{k'}^?, H_{\ell'}^?, S_{p_{i'},q_{i'},r_{i'}}, B'_2, \vec{K}, S_{p_m,q_m,r_m}, C_1, (H_k^?,H_\ell^?,S_{p_i,q_i,r_i}) \BS \underline{x} \to x.
$$
The sentinel $S_{p_{i'},q_{i'},r_{i'}}$, again, blocks other choices for the principal $x$. Applying Lemma~\ref{Lm:decomposition} once more:
\begin{align*}
& H_{k'}^?, H_{\ell'}^?, S_{p_{i'},q_{i'},r_{i'}}, B'_2, \vec{K}, S_{p_m,q_m,r_m}, C_1  = \Psi_1, \Psi_2, \Psi_3;\\
& \vdash \Psi_1 \to H_k^?; \\
& \vdash \Psi_2 \to H_\ell^?;\\
& \vdash \Psi_3 \to S_{p_i,q_i,r_i}.
\end{align*}

The only two formulae in the antecedents, whose tops could potentially be $p_i$, $q_i$, or $r_i$, are the sentinels $S_{p_{i'},q_{i'},r_{i'}}$ and
$S_{p_m,q_m,r_m}$. By Lemma~\ref{Lm:sentinel}, we get $m = i$ and $\Psi_3 = S_{p_m,q_m,r_m}$. 
Therefore, $C_1 = \Lambda$.

Consider several subcases:

{\em Subcase 1.1:} there are neither $H_k$, nor $H_\ell$. Then $\Psi_1 = \Psi_2 = \Lambda$, which is not the case, since $\Psi_1$ or $\Psi_2$
should include at least the other sentinel, $S_{p_{i'},q_{i'},r_{i'}}$. 

{\em Subcase 1.2:} there is only $H_k$, but not $H_\ell$, thus, $\Psi_2 = \Lambda$. Then we have 
\begin{align*}
& \Psi_1 = H_{k'}^?, H_{\ell'}^?, S_{p_{i'},q_{i'},r_{i'}}, B'_2, \vec{K};\\
& \vdash \Psi_1 \to H_k.
\end{align*} Let us check that
$B'_2, \vec{K}$ satisfies the conditions for $\Phi$ in Lemma~\ref{Lm:badcase2}. Indeed, tops of formulae in $B'_2$ are $w$ or $x$; tops of $K_i$
are $z$, and they are of the form $z \SL \Gamma$. Finally, since $B'_2$ is not empty, the first formula in our $\Phi$ is from $B'_2$, and thus does not have top $z$.
By Lemma~\ref{Lm:badcase2}, $\not\vdash H_{k'}^?, H_{\ell'}^?, S_{p_{i'},q_{i'},r_{i'}}, \Phi \to H_k$. Contradiction.

{\em Subcase 1.3:} there are both $H_k$ and $H_\ell$. Then 
$$
\Psi_1, \Psi_2 = H_{k'}^?, H_{\ell'}^?, S_{p_{i'}, q_{i'},r_{i'}}, B'_2, \vec{K}.
$$
Take the one of $\Psi_1$ and $\Psi_2$ which includes $S_{p_{i'},q_{i'},r_{i'}}$. If it also includes the first formula of $B'_2$, then 
the claim for it violates Lemma~\ref{Lm:badcase2}, exactly as in the previous subcase. Otherwise 
$\Psi_2 = B'_2, \vec{K}$, and $\Psi_2 \to H_\ell$ is not derivable by Lemma~\ref{Lm:badcase}.

\vskip 5pt
{\em Case 2.} This is the fruitful case.
$$
\vdash B_2, K_{i_1}, \ldots, K_{i_n}, S_{p_m,q_m,r_m}, C_1, (H_k^?,H_\ell^?,S_{p_i,q_i,r_i}) \BS \underline{x} \to x.
$$
Decomposition (Lemma~\ref{Lm:decomposition}) yields
\begin{align*}
& B_2, K_{i_1}, \ldots, K_{i_n}, S_{p_m,q_m,r_m}, C_1 = \Psi_1, \Psi_2, \Psi_3; \\
& \vdash \Psi_1 \to H_k^?;\\
& \vdash \Psi_2 \to H_\ell^?;\\
& \vdash \Psi_3 \to S_{p_i,q_i,r_i}.
\end{align*}
By Lemma~\ref{Lm:sentinel}, $m = i$ and $\Psi_3$ should be $S_{p_m,q_m,r_m}$. 
Thus, $C_1 = \Lambda$ and $B_2, K_{i_1}, \ldots, K_{i_n} = \Psi_1, \Psi_2$.

Let us show that $B_2$ is empty. If there are no $H_k$ and $H_\ell$, this is trivial, since in this case $\Psi_1 = \Psi_2 = \Lambda$.
Let there be $H_k$ and suppose that $B_2$ is not empty. Then the first formula of $B_2$ is the first formula of $\Psi_1$ ($\Psi_1$ is not
empty, since $\not\vdash \Lambda \to H_k$).  Recall that $H_k = z \SL (S_{p_k,q_k,r_k}, z)$ and invert $(\to\SL)$:
$$
\vdash \Psi_1, S_{p_k,q_k,r_k}, z \to z.
$$
The principal occurrence of $z$ should be in one of the $K_i$, and this formula, being of the form $z \SL \Gamma$, should be the leftmost one.
However, the leftmost formula of $\Psi_1$ is from $B_2$ and does not have top $z$. Contradiction.

Thus, $B_2 = \Lambda$, and we have $\Psi_1 = K_{i_1}, \ldots, K_{i_{n'}}$, $\Psi_2 = K_{i_{n'+1}}, \ldots, K_{i_n}$. By induction hypothesis,
$\vdash \Psi_1 \to H_k$ yields derivability of $a_{i_1} \ldots a_{i_{n'}}$ from $N_k$ and 
$\vdash \Psi_2 \to H_\ell$ yields derivability of $a_{i_{n'+1}} \ldots a_{i_n}$ from $N_\ell$ in the context-free grammar $\Gc$. Finally,
since $A_{j;m,k^?,\ell^?} \in \Uc_j$ (recall that $m=i$), in $\Gc$ we have the production rule needed to finish the derivation:
$$N_m \Rightarrow a_j N_k^? N_\ell^?.$$
\end{proof}

Theorem~\ref{Th:SafiullinE} immediately follows from Lemma~\ref{Lm:key}. The necessary $\LL(\BS,\SL)$-grammar is constructed as follows:
for each $a_j \in \Sigma$ let $a_j \rhd K_j$ (this type assignment is unique) and let $H = H_0$.

Notice that our construction also works with Lambek's restriction, thus subsuming Safiullin's original result (Theorem~\ref{Th:Safiullin}).

\section{Notes on Specific Classes of Models}\label{S:compl}

In this section we consider theories of three specific classes of residuated Kleene lattice, namely, language, regular language, and
relational models. Though, in the language with $\vee$ and $\wedge$, these theories are strictly greater than $\ACTomega$,~\citet{BuszkoRelMiCS}
proves $\Pi_1^0$-hardness for them also. We propagate Buszkowski's complexity results to the corresponding classes of residuated monoids with iteration,
getting rid of additive operations. The completeness issue, 
{\em i.e.,} the question whether these theories in the language without $\vee$ and $\wedge$ coincide with $\Lomega$, is still open.

Language models, or L-models for short, are interpretations of the Lambek calculus and its extensions on the algebra of formal
languages over an alphabet $\Sigma$ (see~\ref{S:intro}).

We also consider a more specific class of L-models, where all variables are interpreted by regular languages.
Since the class of regular languages is closed under all operations we consider, interpretations of arbitrary formulae
are also going to be regular. Following~\citet{BuszkoRelMiCS}, we call these models REGLAN-models.

The second class of interpretations are relational models (R-models). In R-models, formulae are interpreted as binary relations on
a set $W$, {\em i.e.,} subsets of $W \times W$. Operations are defined as follows:
\begin{align*}
& R \BS S = \{ \langle y,z \rangle \in W \times W \mid (\forall \langle x, y \rangle \in R) \, \langle x, z \rangle \in S \};\\
& S \SL R = \{ \langle x,y \rangle \in W \times W \mid (\forall \langle y, z \rangle \in R) \, \langle x, z \rangle \in S \};\\
& R \cdot S = R \circ S = \{ \langle x, z \rangle  \in W \times W \mid (\exists y \in W)\, \langle x,y \rangle \in R \mbox{ and }
\langle y, z \rangle \in S \}.
\end{align*}
(For simplicity, we consider only ``square'' relational models, where any pair $\langle x,y \rangle$ could belong to a relation. There
is also a broader class of ``relativised'' relational models, where all relations are subsets of a ``universal'' 
 relation $U$~\citep{AndrekaMikulas}, which is reflexive and transitive. Relativisation
alters the definition of division operations. Relativised R-models are necessary for the Lambek calculus with Lambek's restriction (\ref{S:Lrestr}), where
one drops the reflexivity condition on $U$.)

In all these classes of models additive connectives are interpreted set-theoretically, as union and intersection.

Unfortunately, no completeness results are known for $\Lomega$ w.r.t. these specific classes of models. Moreover, even adding
only one additive connective, conjunction $\wedge$, yields incompleteness w.r.t. all three classes of models~\citep{Kuzn2018AiML}. For action logic
in whole, incompleteness is connected with the distributivity principle, $(A \vee C) \wedge (B \vee C) \to (A \wedge B) \vee C$. This principle is
true in all models in which $\vee$ and $\wedge$ are interpreted set-theoretically, as union and intersection, but is not derivable in $\ACTomega$,
cf.~\citet{OnoKomori1985}. Since $A^*$ is essentially infinite disjunction ($\One \vee A \vee A^2 \vee \ldots$), incompleteness also propagates to
the fragment without explicit $\vee$. On the other hand, $\LL$ is complete w.r.t. L-models~\citep{PentusFmonov} and R-models~\citep{AndrekaMikulas}.
For REGLAN-models, completeness is an open problem, but there are some partial results which we shall use later. 

Completeness for $\Lomega$ itself is an open problem, for any of these three interpretations. Thus, the inequational theories of L-models,
R-models, and REGLAN-models, in the language with $\BS$, $\SL$, $\cdot$, and $\KStar$, could possibly be different from the set of theorems
of $\Lomega$. We denote these theories by $\ThM(\LMod)$, $\ThM(\RMod)$, and $\ThM(\REGLANMod)$, respectively.

In the bigger language of $\ACTomega$, the corresponding theories are denoted by $\ThA(\mathcal{M})$, where $\mathcal{M}$ is $\LMod$, $\RMod$, or
$\REGLANMod$. These theories are definitely different from $\ACTomega$ itself, due to the distributivity principle. However,~\citet{BuszkoRelMiCS} manages to
propagate his $\Pi_1^0$-hardness result to these these theories also:
\begin{thm}[Buszkowski]
Theories $\ThA(\LMod)$, $\ThA(\RMod)$, and $\ThA(\REGLANMod)$ are $\Pi_1^0$-hard.
\end{thm}

Due to the semantic definition of these theories, however, the ``out-of-the-box'' upper bound appears to be quite high.
For $\ThA(\LMod)$ and $\ThA(\RMod)$ it is $\Pi_1^1$: indeed, the condition for a sequent to belong to one of these theories starts with a second-order quantifier
``for any model,'' followed by an arithmetically formulated truth condition. No better upper complexity bounds are known. For REGLAN-models,
 however, the situation is different~\citep{BuszkoRelMiCS}. Such a model is essentially a bunch of regular expressions encoding the languages which
 interpret variables (for a given sequent, the set of its variables is finite), and this bunch can be encoded by a natural number.
  This arithmetises the quantifier over all models and gives $\Pi_1^0$ as the upper bound. 
  Thus, the upper bound matches the lower one: $\ThA(\REGLANMod)$ is $\Pi_1^0$-complete.

Buszkowski's method of proving $\Pi_1^0$-hardness is essentially based on the fact that the fragment which is  really needed to prove
$\Pi_1^0$-hardness of $\ACTomega$ is in fact complete w.r.t. all three classes of models. Here we formulate this fragment explicitly,
prove its completeness and thus obtain $\Pi_1^0$-hardness for the theories without $\vee$ and $\wedge$, using our Theorem~\ref{Th:main}.

Let us call a formula {\em *-external,} if no $\cdot$ or $\KStar$ in it occurs within the scope of $\BS$ or $\SL$. More formally, the class of
*-external formulae is defined recursively as follows:
\begin{itemize}
\item any formula in the language of $\BS$ and $\SL$ is *-external;
\item if $A$ and $B$ are *-external, then so is $A \cdot B$;
\item if $A$ is *-external, then so is $A^*$.
\end{itemize} 
A sequent $A_1, \ldots, A_n \to B$ is called *-external, if all $A_i$ are *-external and $B$ is a formula in the language of $\BS$ and $\SL$.

By subformula property, all sequents in a cut-free derivation of a *-external sequent are also *-external. Thus, the notion of *-externality induces a
fragment of $\Lomega$---and we show that this fragment is simultaneously $\Pi_1^0$-hard and complete w.r.t. L-models, R-models, and REGLAN-models.
The first claim immediately follows from the form of the sequent used in the proof of Theorem~\ref{Th:main}: it is *-external (recall that
our construction of an $\LL$-grammar with unique type assignment, Theorem~\ref{Th:SafiullinE}, uses only two divisions, $\BS$ and $\SL$).

For the second claim, we present a version of Palka's *-elimination via approximation, which reduces derivability of *-external sequents
 in $\Lomega$ to derivability in $\LL(\BS,\SL)$,
without using $\vee$. For each *-external formula $A$ we define the set of its {\em instances,} $\Inst(A)$, where subformulae of the
form $B^*$ are replaced by 
concrete numbers of $B$'s (since Kleene stars could be nested, these $A$'s should also be replaced by instances, and therefore become different), 
and $\cdot$'s are replaced by metasyntactic commas (thus, instances are sequences of formulae).
The set of instances is defined by induction on the construction of a *-external formula: 
\begin{itemize}
\item if $A$ is a formula in the language of $\BS$ and $\SL$, then $\Inst(A) = \{A\}$;
\item $\Inst(A \cdot B) = \{ \Gamma, \Delta \mid \Gamma \in \Inst(A), \Delta \in \Inst(B) \}$;
\item $\Inst(A^*) = \{ \Gamma_1, \ldots, \Gamma_n \mid n \ge 0, \Gamma_i \in \Inst(A) \}$.
\end{itemize}

A typical example of the notion of the set of instances is given in the proof of Theorem~\ref{Th:main}: for $A = (A_1^+ \cdot A_2^+)^+$ the set
of instances is $$\{ A_1^{n_1}, A_2^{m_2}, A_1^{n_2}, A_2^{m_2}, \ldots, A_1^{n_k}, A_2^{m_k} \mid k, n_1, m_1, n_2, m_2, \ldots, n_k, m_k \ge 1 \}.$$

\begin{lem}\label{Lm:external}
A *-external sequent $A_1, \ldots, A_n \to B$ is derivable in $\Lomega$ if and only if $\Pi_1, \ldots, \Pi_n \to B$ is derivable in $\LL(\BS,\SL)$ for
any $\Pi_1 \in \Inst(A_1)$, \ldots, $\Pi_n \in \Inst(A_n)$, i.e., all its instances are derivable.
\end{lem}

\begin{proof}
Essentially the same as Lemma~\ref{Lm:externalsimple}: the ``if'' part goes by applying $(\cdot\to)$ and $(\KStar\to)_\omega$;
the ``only if'' one goes by their inversion.
\end{proof}

\begin{lem}\label{Lm:external2}
For any *-external formula $A$ and any $\Pi \in \Inst(A)$ the sequent $\Pi \to A$ is derivable in $\Lomega$.
\end{lem}

\begin{proof}
Induction on the structure of $A$. If $A$ is in the language of $\BS$ and $\SL$, then $\Pi = A$ and $\Pi \to A$ is an axiom.
If $A = B \cdot C$, then $\Pi = \Gamma, \Delta$, where $\Gamma \in \Inst(B)$ and $\Delta \in \Inst(C)$. By induction hypothesis,
$\Gamma \to B$ and $\Delta \to C$ are derivable; then $\Pi \to A$ is derivable by application of $(\to\cdot)$.
If $A = B^*$, then $\Pi = \Pi_1, \ldots, \Pi_n$, where $\Pi_i \in \Inst(B)$. By induction hypothesis, all sequents $\Pi_i \to B$ are derivable,
and thus $\Pi \to A$ is derivable by application of $(\to\KStar)_n$.
\end{proof}

Now we are ready to prove completeness.
\begin{thm}\label{Th:partialcompleteness}
A *-external sequent is derivable in $\Lomega$ if and only if it is true in
all L-models, or all REGLAN-models, or all R-models.
\end{thm}

\begin{proof}
The ``only if'' part follows from the general soundness theorem of $\Lomega$ w.r.t. arbitrary RKLs.

For the ``if'' part, we first recall completeness results for $\LL(\BS,\SL)$:
\begin{itemize}
\item for L-models, completeness of $\LL$ was proved by~\citet{PentusFmonov}; here we can actually use a simpler result
by~\citet{Buszkowski1982compat} for the product-free fragment;
\item for REGLAN-models, completeness follows from the fact that $\LL(\BS,\SL)$ is complete even w.r.t. a narrower class of L-models,
in which variables are interpreted by cofinite languages~\citep{Buszkowski1982decision,Sorokin2012};
\item for R-models, completeness was proved by~\citet{AndrekaMikulas}.
\end{itemize}

Now let a *-external sequent $A_1, \ldots, A_n \to B$ be true in all models of one of the classes: $\LMod$, $\REGLANMod$, or $\RMod$. Let
$\Pi_1 \in \Inst(A_1)$, \ldots, $\Pi_n \in \Inst(A_n)$. By Lemma~\ref{Lm:external2}, $\Pi_i \to A_i$ are derivable in $\Lomega$, and by soundness they
are true in all models of the specified class. Thus, $\Pi_1, \ldots, \Pi_n \to B$ is also generally true. This is a sequent in the language of $\BS$ and $\SL$,
and by completeness results mentioned above it is derivable in $\LL(\BS,\SL)$. Now, since instances $\Pi_1$, \ldots, $\Pi_n$ were taken arbitrarily,
by Lemma~\ref{Lm:external} we conclude that the original sequent $A_1, \ldots, A_n \to B$ is derivable in $\Lomega$.
\end{proof}

Now we can prove the $\Pi_1^0$-hardness results.
\begin{thm}
Theories $\ThM(\LMod)$, $\ThM(\REGLANMod)$, and \\ $\ThM(\RMod)$ are $\Pi_1^0$-hard.
\end{thm}

\begin{proof}
If we consider only *-external sequents, the corresponding fragments of these theories are equal to the one of $\Lomega$. 
As noticed above, the latter is $\Pi_1^0$-hard.
\end{proof}

Like the system with additives, $\ThM(\REGLANMod)$ also enjoys a $\Pi_1^0$ upper complexity bound, and thus is $\Pi_1^0$-complete. For two other theories, 
$\ThM(\LMod)$ and $\ThM(\RMod)$, the best known upper bound is only $\Pi_1^1$.

\section{Conclusion and Future Work}\label{S:conclusion}

We have proved $\Pi_1^0$-hardness for $\Lomega$, the Lambek calculus with Kleene star, and the corresponding inequational theories
of the algebras of languages, regular languages, and binary relations. These results strengthen results by Buszkowski
for the corresponding systems extended with additive connectives, $\vee$ and $\wedge$.
The crucial component of our proof is the construction of an $\LL(\BS,\SL)$-grammar with unique type assignment for a context-free grammar without the
empty word.

Let us briefly survey the questions which are still open in this area.

First,  the completeness issue of $\Lomega$ w.r.t. L-models, REGLAN-models, and
R-models is still open. 

Second, it is interesting to characterise the class of languages that can be generated by $\Lomega$-grammars---in 
particular, whether such a grammar could generate a $\Pi_1^0$-hard language.
Notice that the complexity of the calculus in whole could be greater than that of concrete languages generated by grammars
based on this calculus. For example, the original Lambek calculus is NP-complete~\citep{Pentus2006}, while all languages generated
by Lambek grammars are context-free~\citep{PentusCF}, and therefore decidable in polynomial time. 

Third, we do not yet know whether any 
context-free language without the empty word can be generated by a Lambek grammar with unique type assignments with only one division.
(If one drops the uniqueness condition, then already Gaifman's construction yields such a grammar.) A positive answer to this question would
yield $\Pi_1^0$-hardness for the fragment of $\Lomega$ in the language of $\cdot$, $\SL$, and $\KStar$. 

Fourth, there is a small question on $\LL$-grammars
with unique type assignment. Grammars constructed in this article are not capable of generating the empty word (because context-free
grammars in Greibach form cannot generate it), whereas in general $\LL$ allows empty
antecedents and therefore $\LL$-grammars could potentially generate grammars with the empty word. Moreover, in the case without uniqueness condition there
exists a construction that transforms context-free grammars with the empty word into $\LL$-grammars~\citep{Kuznetsov2012IGPL}.
The question whether all context-free languages with the empty word can be generated by $\LL$-grammars with unique type assignment is still open.

Finally, as noticed by one of the referees, in our $\Pi_1^0$-hardness proof the product ($\cdot$) operation is used only once,
in $(A_1^+ \cdot A_2^+)^+ \to H$, while $A_1$, $A_2$, and $H$ are product-free. The question is whether it is possible to get rid of
the product completely and prove $\Pi_1^0$-hardness for $\ACTomega(\BS,\SL,\KStar)$. We conjecture that this could be done by
the ``pseudo-double-negation'' trick~\citep{Buszkowski2007,KanKuzSce2019WoLLICpspace}. Namely, $(A_1^+ \cdot A_2^+)^+ \to H$ is probably
equiderivable with $\bigl( b \SL ((b \SL A_2^+) \SL A_1^+) \bigr)^+ \to b \SL (b \SL H)$, for a fresh variable $b$. Establishing this
equiderivability (and, thus, $\Pi_1^0$-hardness of the product-free fragment of $\Lomega$) is left for future research.

\subsection*{Acknowledgments}
The author is grateful to the organisers and participants of WoLLIC 2017 and MIAN--POMI 2018 Winter Session on Mathematical Logic for
fruitful discussions and friendly atmosphere. The author is also indebted to the anonymous referees for valuable comments and suggestions.
Being a Young Russian Mathematics award winner, the author would like to thank its sponsors and jury for this high honour.

\subsection*{Funding}
The work is supported by the Russian Science Foundation under grant 16-11-10252.

\vskip 10pt
\bibliographystyle{rsl}
\bibliography{Lstar_Pi1}
\vspace*{10pt}

\address{STEKLOV MATHEMATICAL INSTITUTE\\
\hspace*{9pt}RUSSIAN ACADEMY OF SCIENCES\\
\hspace*{18pt}8 GUBKINA ST., 119991 MOSCOW, RUSSIA\\
{\it E-mail:} sk@mi-ras.ru}
\end{document}